\documentclass[12pt,bezier]{article}

\usepackage{amssymb,amsmath,amsthm,latexsym,tikz,url,float,subfig,caption,graphicx,pgfplots}
 \usetikzlibrary{matrix}
\newtheorem{theorem}{Theorem}[section]

\newtheorem{conjecture}[theorem]{Conjecture}
\newtheorem{problem}[theorem]{Problem}
\newtheorem{lemma}[theorem]{Lemma}
\newtheorem{remark}[theorem]{Remark}
\newcommand{\la}{\lambda}

\newcommand{\Fc}{{\cal F}_{\rm cub}}
\newcommand{\Fq}{{\cal F}_{\rm qua}}
\newcommand{\de}{\Delta}
\newcommand{\ga}{\Gamma}

\newcommand{\x}{{\bf x}}
\newcommand{\Z}{\mathbb{Z}}
\newcommand{\Rbb}{\mathbb{R}}
\newcommand{\bG}{\boldsymbol{G}}
\newcommand{\bde}{\boldsymbol{\Delta}}
\newcommand{\bga}{\boldsymbol{\Gamma}}

\tikzstyle{vertex}=[circle, draw, inner sep=0pt, minimum size=4pt]
\newcommand{\vertex}{\node[vertex]}
\usepackage{color}

\begin{document}

\title{Gap sets for the spectra of regular graphs with minimum spectral gap}
\author{ Maryam Abdi$^{\,\rm a}$\qquad Ebrahim Ghorbani$^{\,\rm a,b}$
\\[.3cm]
	{\sl $^{\rm a}$Department of Mathematics, K. N. Toosi University of Technology,}\\
{\sl P. O. Box 16765-3381, Tehran, Iran}\\
{\sl $^{\rm b}$Department of Mathematics, University of Hamburg, }\\
{\sl Bundesstra\ss e 55 (Geomatikum), 20146 Hamburg, Germany }}
\maketitle
\footnotetext{{\em E-mail Addresses}: {\tt m.abdi@email.kntu.ac.ir} (M. Abdi), {\tt e\_ghorbani@ipm.ir} (E. Ghorbani) }

\begin{abstract}
Following recent work by Koll\'{a}r and Sarnak, we study gaps in the spectra of large connected  cubic and quartic graphs with minimum spectral gap. We focus on two sequences of graphs, denoted $\de_n$ and $\ga_n$ which are more `symmetric' 
 compared to the other graphs in these two families, respectively. We prove that
 $(1,\sqrt{5}]$ is a gap interval for $\de_n$, and
$[(-1+\sqrt{17})/2,3]$ is a gap interval for $\ga_n$.
	We conjecture that these two are indeed maximal gap intervals.
	As a by-product, we show that the eigenvalues of $\de_n$ lying in the interval $[-3,-\sqrt{5}]$ (in particular, its minimum eigenvalue) converge to $(1-\sqrt{33})/2$  and the eigenvalues of $\ga_n$ lying in the interval $[-4,-(1+\sqrt{17})/2]$  (and in particular, its minimum eigenvalue) converge to $1-\sqrt{13}$ as $n$ tends to infinity.
	The proofs of the above results heavily depend on the following property which can be of independent interest:
 with few exceptions, all the eigenvalues of connected cubic and quartic graphs with minimum spectral gap  are simple.
	
\noindent {\bf Keywords:}  Spectral gap, Gap set in spectrum, Cubic graph, Quartic graph \\[.1cm]
\noindent {\bf AMS Mathematics Subject Classification (2010):}   05C50
\end{abstract}

\section{Introduction}
Let $G$ be a simple graph.
 The eigenvalues of  $G$ are  defined as the eigenvalues of its adjacency matrix.
The {\em spectrum} of $G$ is the list of its eigenvalues. The difference between the two largest eigenvalues of $G$ is called the {\em spectral gap} of $G$.
 Aldous and Fill (see \cite[p.~217]{aldous2002reversible}) posed a conjecture on the maximum relaxation time of the random walk in regular graphs.  This conjecture can be rephrased in terms of the spectral gap as follows: the spectral gap of a connected  $r$-regular graph on $n$ vertices is at least
$(1+o(1))\frac{2r\pi^2}{3n^2}$. Brand, Guiduli, and Imrich \cite{Imrich} (confirming a conjecture by L. Babai (see \cite{Guiduli})) determined the structure of connected cubic (i.e. $3$-regular) graphs with minimum spectral gap. For every even $n\ge10$, such a graph is proved to be unique. We denote this family of graphs by $\Fc$ (see Section~\ref{sec:pre}
for  the precise description).
 Abdi, Ghorbani and Imrich \cite{AbGhIm} showed that the spectral gap of the graphs of $\Fc$ is $(1+o(1))\frac{2\pi^2}{n^2}$, confirming the Aldous--Fill conjecture for $r=3$. Continuing this line of investigation,
Abdi and Ghorbani \cite{AbGh} gave a `near' complete characterization of the graphs with the minimum spectral gap among the connected quartic  (i.e. $4$-regular) graphs of a fixed order. They showed that such graphs belong to a specific family, here denoted by $\Fq^*$ (see Section~\ref{sec:pre}
for the precise description). Based on this result, they established  the Aldous--Fill conjecture for $r=4$. They further conjectured that the graphs in a subfamily $\Fq$ of $\Fq^*$  give the unique {\em minimal} graph for every order $n$.

Koll\'{a}r and Sarnak \cite{Sarnak} studied gaps in the spectra of large finite cubic graphs. 
An interval in the real line that contains no eigenvalues of $G$ is called  a {\em gap interval} for $G$.
It is known that the gap interval $(2\sqrt2, 3)$ achieved in cubic Ramanujan graphs  is {\em maximal} meaning that for any $\epsilon>0$, the interval $(2\sqrt2-\epsilon, 3)$ contains some eigenvalues of some cubic Ramanujan graphs. Also the gap interval $[-3,-2)$ achieved in cubic line graphs is maximal.  Koll\'{a}r and Sarnak gave constraints on spectra in $[-3,3]$ which are maximally gapped and constructed examples which achieve these bounds. Among other results, they also showed that every point in $[-3, 3)$ can be gapped (meaning that every point is contained in an open interval which is a gap set) by planar cubic graphs.  They also posed the similar spectral gap questions for general regular graphs.
\begin{problem}[Koll\'{a}r and Sarnak \cite{Sarnak}]\rm
Investigate the spectral gap questions more generally for $r$-regular graphs ($r > 3$).
\end{problem}
 This is the main motivation in this work. Our goal is to determine what gaps can be achieved by large connected cubic and quartic  graphs with minimum spectral gap. More precisely, our focus will be on the graphs in $\Fc$ with orders $\equiv2\pmod4$ and the graphs in $\Fq$ with orders $\equiv1\pmod5$ which are denoted by $\de_n$ and $\ga_n$, respectively. These graphs are more `symmetric' than the other graphs of their respective families. We prove that:
\begin{itemize}
	\item $(1,\sqrt{5}]$ is a gap interval for the graphs $\de_n$;
	\item  $[(-1+\sqrt{17})/2,3]$ is a gap interval for the graphs $\ga_n$.
\end{itemize}
We conjecture that the above two intervals are indeed maximal gap intervals.
As a by-product, we obtain the following results. Here $\rho(G)$ denotes the smallest eigenvalue of a graph $G$.
\begin{itemize}
	\item The eigenvalues of $\de_n$ lying in the interval $\left[-3,-\sqrt{5}\right]$ converge to $(1-\sqrt{33})/2$ as $n$ tends to infinity. In particular, 	 $\underset{n\to \infty}{\lim}\rho(\de_n)=(1-\sqrt{33})/2$.
	\item The eigenvalues of $\ga_n$ lying in the interval $[-4,-(1+\sqrt{17})/2]$ converge to $1-\sqrt{13}$ as $n$ tends to infinity. In particular, $\underset{n\to \infty}{\lim}\rho(\ga_n)=1-\sqrt{13}$.
\end{itemize}
The proofs of the above results heavily depend on the following properties on the simplicity of the eigenvalues of the graphs in $\Fc$ and $\Fq^*$ which can be of independent interest.
\begin{itemize}
	\item  Let $G$ be a graph in $\Fc$ with order  $n\ge10$.
	If  $n\equiv2\pmod4$, then all the eigenvalues of $G$ except $-1$ and $0$ are simple. If $n\equiv 0\pmod4$, then the only non-simple eigenvalue of $G$ is $-1$.
	\item Non-simple  eigenvalues of the graphs in $\Fq^*$ of order $n\geq 11$ belong to
	$$\left\{-2,0, \pm 1, -1\pm\sqrt{2}, (-1\pm\sqrt{5})/2\right\}.$$
\end{itemize}

In passing, we remark that several results on gap sets for specific families of graphs can be found in the literature. Jacobs {\em et al.} \cite{jtt0} showed that threshold graphs (that is graphs with no induced $P_4$, $C_4$ nor $2K_2$) have no eigenvalues in $(-1, 0)$.
This result was extended by Ghorbani \cite{ghCo} who showed that a graph $G$ is a cograph  (i.e. a $P_4$-free graph) if and only if no induced subgraph of $G$ has an eigenvalue in the interval $(-1,0)$. Improving the aforementioned result of Jacobs {\em et al.}, Ghorbani \cite{ghTh} proved that besides $0,-1$,  threshold  graphs  have no eigenvalues in the maximal gap interval $\left[(-1-\sqrt{2})/2,\,(-1+\sqrt{2})/2\right]$. This was first conjectured by Aguilar  {\em et al.} \cite{alps}.  In \cite{ghCo}, it is also proved that bipartite $P_5$-free graphs  have no eigenvalues in the intervals $(-1/2,0)$ and $(0, 1 / 2)$.
The result of \cite{ghTh} is extended  to the case that if the `generating sequence' of a threshold graph is given by Andeli\'c {\em et al.} \cite{Andelic}  and to the spectrum  of distance matrix by Alazemi {\em et al.}  \cite{Alazemi}.

The rest of the paper is organized as follows. In Section~\ref{sec:pre}, we present the precise description of the graphs in $\Fc$ and $\Fq^*$, and recall some basic facts about graph eigenvalues. In Section~\ref{sec:Floque-Bloch}, we give a brief overview of Floque--Bloch transform where we apply it to the two infinite graphs that correspond naturally to the sequences $\de_n$ and $\ga_n$ and obtain their bands of continuous spectra. In Section~\ref{sec:simplicity}, we prove the results on simplicity of the eigenvalues of the graphs in these two families.
  This will be used in Section~\ref{sec:gap}, to establish the results on the gap intervals and other properties of the spectra of the graphs $\de_n$ and $\ga_n$.

\section{Preliminaries}\label{sec:pre}

In this section, we give the precise description of the families $\Fc$ and $\Fq^*$ and
recall some basic facts concerning graph eigenvalues.

\subsection{Cubic and quartic graphs with minimum spectral gap}
 L. Babai (see \cite{Guiduli}) posed a conjecture on the structure of the connected cubic graph with minimum  spectral gap.
Brand, Guiduli, and Imrich \cite{Imrich} confirmed the conjecture by proving that
for any even $n\ge10$, the $n$-vertex graph given in Figure~\ref{fig:mincubic}
 is the unique graph with minimum spectral gap among connected cubic graphs of order $n$. Note that cubic graphs always have even order.
  We denote the family consisting of these `minimal' cubic graphs by $\Fc$.

 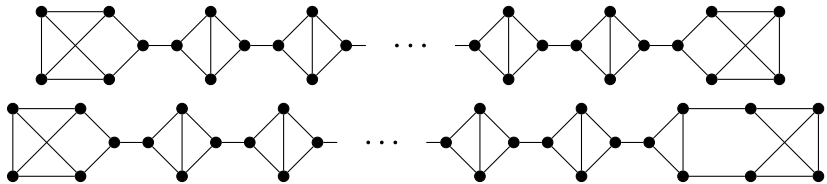
\begin{figure}[H]
	\centering
	\begin{tikzpicture}[scale=0.9]
		\vertex[fill] (1) at (0,-.5) [] {};
		\vertex[fill] (2) at (0,.5) [] {};
		\vertex[fill] (3) at (1,-.5) [] {};
		\vertex[fill] (4) at (1,.5) [] {};
		\vertex[fill] (5) at (1.5,0) [] {};
		\vertex[fill] (6) at (2,0) [] {};
		\vertex[fill] (7) at (2.5,.5) [] {};
		\vertex[fill] (8) at (2.5,-.5) [] {};
		\vertex[fill] (9) at (3,0) [] {};
		\vertex[fill] (10) at (3.5,0) [] {};
		\vertex[fill] (11) at (4,.5) [] {};
		\vertex[fill] (12) at (4,-.5) [] {};
		\vertex[fill] (13) at (4.5,0) [] {};
		\vertex[fill] (22) at (6.4,0) [] {};
		\vertex[fill] (23) at (6.9,.5) [] {};
		\vertex[fill] (24) at (6.9,-.5) [] {};
		\vertex[fill] (25) at (7.4,0) [] {};
		\vertex[fill] (26) at (7.9,0) [] {};
		\vertex[fill] (27) at (8.4,.5) [] {};
		\vertex[fill] (28) at (8.4,-.5) [] {};
		\vertex[fill] (29) at (8.9,0) [] {};
		\vertex[fill] (34) at (10.9,-.5) [] {};
		\vertex[fill] (33) at (10.9,.5) [] {};
		\vertex[fill] (32) at (9.9,-.5) [] {};
		\vertex[fill] (31) at (9.9,.5) [] {};
		\vertex[fill] (30) at (9.4,0) [] {};
		\tikzstyle{vertex}=[circle, draw, inner sep=0pt, minimum size=1pt]
		\vertex[fill] (15) at (5.25,0) [] {};
		\vertex[fill] (17) at (5.45,0) [] {};
		\vertex[fill] (19) at (5.65,0) [] {};
		\tikzstyle{vertex}=[circle, draw, inner sep=0pt, minimum size=0pt]
		\vertex[] (14) at (4.8,0) [] {};
		\vertex[] (21) at (6.1,0) [] {};
		\path
		(1) edge (2)
		(1) edge (3)
		(1) edge (4)
		(2) edge (3)
		(2) edge (4)
		(3) edge (5)
		(4) edge (5)
		(5) edge (6)
		(6) edge (7)
		(6) edge (8)
		(7) edge (8)
		(7) edge (9)
		(8) edge (9)
		(9) edge (10)
		(10) edge (11)
		(10) edge (12)
		(11) edge (12)
		(11) edge (13)
		(12) edge (13)
		(13) edge (14)
		(21) edge (22)
		(22) edge (23)
		(22) edge (24)
		(23) edge (24)
		(23) edge (25)
		(24) edge (25)
		(25) edge (26)
		(26) edge (27)
		(26) edge (28)
		(27) edge (28)
		(27) edge (29)
		(28) edge (29)
		(29) edge (30)
		(30) edge (31)
		(30) edge (32)
		(31) edge (33)
		(31) edge (34)
		(32) edge (33)
		(32) edge (34)
		(33) edge (34);
	\end{tikzpicture}
	\vspace{.2cm} \\
	\begin{tikzpicture}[scale=0.9]
		\vertex[fill] (1) at (0,-.5) [] {};
		\vertex[fill] (2) at (0,.5) [] {};
		\vertex[fill] (3) at (1,-.5) [] {};
		\vertex[fill] (4) at (1,.5) [] {};
		\vertex[fill] (5) at (1.5,0) [] {};
		\vertex[fill] (6) at (2,0) [] {};
		\vertex[fill] (7) at (2.5,.5) [] {};
		\vertex[fill] (8) at (2.5,-.5) [] {};
		\vertex[fill] (9) at (3,0) [] {};
		\vertex[fill] (10) at (3.5,0) [] {};
		\vertex[fill] (11) at (4,.5) [] {};
		\vertex[fill] (12) at (4,-.5) [] {};
		\vertex[fill] (13) at (4.5,0) [] {};
		\vertex[fill] (22) at (6.4,0) [] {};
		\vertex[fill] (23) at (6.9,.5) [] {};
		\vertex[fill] (24) at (6.9,-.5) [] {};
		\vertex[fill] (25) at (7.4,0) [] {};
		\vertex[fill] (26) at (7.9,0) [] {};
		\vertex[fill] (27) at (8.4,.5) [] {};
		\vertex[fill] (28) at (8.4,-.5) [] {};
		\vertex[fill] (29) at (8.9,0) [] {};
		\vertex[fill] (34) at (10.9,-.5) [] {};
		\vertex[fill] (33) at (10.9,.5) [] {};
		\vertex[fill] (32) at (9.9,-.5) [] {};
		\vertex[fill] (31) at (9.9,.5) [] {};
		\vertex[fill] (30) at (9.4,0) [] {};
		\vertex[fill] (35) at (11.9,-.5) [] {};
		\vertex[fill] (36) at (11.9,.5) [] {};
		\tikzstyle{vertex}=[circle, draw, inner sep=0pt, minimum size=1pt]
		\vertex[fill] (15) at (5.25,0) [] {};
		\vertex[fill] (17) at (5.45,0) [] {};
		\vertex[fill] (19) at (5.65,0) [] {};
		\tikzstyle{vertex}=[circle, draw, inner sep=0pt, minimum size=0pt]
		\vertex[] (14) at (4.8,0) [] {};
		\vertex[] (21) at (6.1,0) [] {};
		\path
		(1) edge (2)
		(1) edge (3)
		(1) edge (4)
		(2) edge (3)
		(2) edge (4)
		(3) edge (5)
		(4) edge (5)
		(5) edge (6)
		(6) edge (7)
		(6) edge (8)
		(7) edge (8)
		(7) edge (9)
		(8) edge (9)
		(9) edge (10)
		(10) edge (11)
		(10) edge (12)
		(11) edge (12)
		(11) edge (13)
		(12) edge (13)
		(13) edge (14)
		(21) edge (22)
		(22) edge (23)
		(22) edge (24)
		(23) edge (24)
		(23) edge (25)
		(24) edge (25)
		(25) edge (26)
		(26) edge (27)
		(26) edge (28)
		(27) edge (28)
		(27) edge (29)
		(28) edge (29)
		(29) edge (30)
		(30) edge (31)
		(30) edge (32)
		(31) edge (32)
		(31) edge (33)
		(32) edge (34)
		(33) edge (35)
		(34) edge (36)
		(33) edge (36)
		(34) edge (35)
		(35) edge (36);
	\end{tikzpicture}
	\caption{The graphs of $\Fc$ of order $n\equiv2\pmod4$ and $n\equiv0\pmod4$, resp.}
	\label{fig:mincubic}
\end{figure}

For every $n\equiv2\pmod4$, we denote the $n$-vertex graph of $\Fc$ by $\de_n$. This is the upper graph of Figure~\ref{fig:mincubic}.

Turning to quartic graphs,
we define the family $\Fq^*$ as follows. Similarly to the graphs in $\Fc$, the graphs in  $\Fq^*$ have  a {\em path-like structure} as illustrated in Figure~\ref{fig:path-like}. Further, their building blocks are those given in Figure~\ref{fig:thm:quartic}; any middle block of $G$ is $M$, and  the left end block of $G$ is one of  $D_1,\ldots,D_5$.  The right end block is the mirror image of one of these five blocks.

\begin{figure}[h!]
	\centering
	\begin{tikzpicture}[scale=.9]
		\tikzstyle{vertex}=[draw, inner sep=0pt, minimum size=0pt]
		\vertex[fill] (1) at (0,0) [label=left:\tiny{}] {};
		\vertex[fill] (2) at (.5,.5) [] {};
		\vertex[fill] (3) at (.5,-.5) [] {};
		\vertex[fill] (4) at (1.5,.5) [] {};
		\vertex[fill] (5) at (1.5,-.5) [] {};
		\vertex[fill] (6) at (2,0) [] {};
		\vertex[fill] (7) at (2.5,.5) [] {};
		\vertex[fill] (8) at (2.5,-.5) [] {};
		\vertex[fill] (9) at (3.5,.5) [] {};
		\vertex[fill] (10) at (3.5,-.5) [] {};
		\vertex[fill] (11) at (4,0) [] {};
		\vertex[fill] (12) at (4.5,.5) [] {};
		\vertex[fill] (13) at (4.5,-.5) [] {};
		\vertex[fill] (14) at (5.5,.5) [] {};
		\vertex[fill] (15) at (5.5,-.5) [] {};
		\vertex[fill] (16) at (6,0) [] {};
		\vertex[fill] (20) at (7.6,0) [] {};
		\vertex[fill] (21) at (8.1,.5) [] {};
		\vertex[fill] (22) at (8.1,-.5) [] {};
		\vertex[fill] (23) at (9.1,.5) [] {};
		\vertex[fill] (24) at (9.1,-.5) [] {};
		\vertex[fill] (25) at (9.6,0) [] {};
		\tikzstyle{vertex}=[circle, draw, inner sep=0pt, minimum size=1pt]
		\vertex[fill] (17) at (6.6,0) [] {};
		\vertex[fill] (18) at (6.8,0) [] {};
		\vertex[fill] (19) at (7,0) [] {};
		\tikzstyle{vertex}=[circle, draw, inner sep=0pt, minimum size=0pt]
		\vertex (s) at (6.2,.2) [label=right:$$] {};
		\vertex (ss) at (6.2,-.2) [label=right:$$] {};
		\vertex (sss) at (7.4,.2) [label=right:$$] {};
		\vertex (ssss) at (7.4,-.2) [label=right:$$] {};
		\path[draw,thick]
		(1) edge (2)
		(1) edge (3)
		(3) edge (5)
		(2) edge (4)
		(4) edge (6)
		(5) edge (6)
		(6) edge (7)
		(6) edge (8)
		(7) edge (9)
		(8) edge (10)
		(10) edge (11)
		(9) edge (11)
		(11) edge (12)
		(11) edge (13)
		(12) edge (14)
		(13) edge (15)
		(14) edge (16)
		(15) edge (16)
		(16) edge (ss)
		(16) edge (s)
		(sss) edge (20)
		(ssss) edge (20)
		(20) edge (21)
		(20) edge (22)
		(21) edge (23)
		(22) edge (24)
		(24) edge (25)
		(23) edge (25);
	\end{tikzpicture}
	\caption{The path-like structure}\label{fig:path-like}
\end{figure}
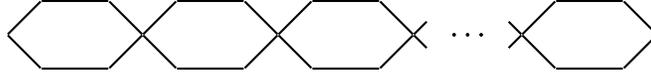

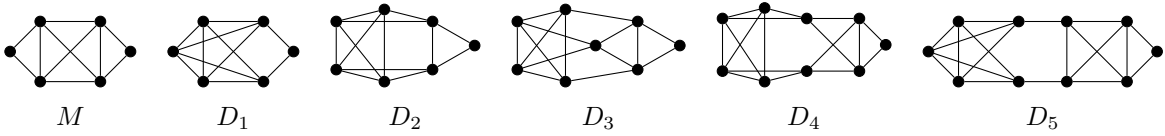
\begin{figure}[h!]
	\captionsetup[subfigure]{labelformat=empty}
	\centering
	\subfloat[$M$]{\begin{tikzpicture}[scale=0.8]
			\vertex[fill] (r) at (0,0) [] {};
			\vertex[fill] (r1) at (.5,.5) [] {};
			\vertex[fill] (r2) at (.5,-.5) [] {};
			\vertex[fill] (r3) at (1.5,.5) [] {};
			\vertex[fill] (r4) at (1.5,-.5) [] {};
			\vertex[fill] (r5) at (2,0) [][] {};
			\path
			(r) edge (r1)
			(r) edge (r2)
			(r1) edge (r2)
			(r1) edge (r3)
			(r1) edge (r4)
			(r2) edge (r3)
			(r2) edge (r4)
			(r3) edge (r4)
			(r5) edge (r4)
			(r3) edge (r5)
			;
	\end{tikzpicture}}
	\quad
	\subfloat[$D_1$]{\begin{tikzpicture}[scale=0.8]
			\vertex[fill] (r) at (0,0) [] {};
			\vertex[fill] (r1) at (.5,.5) [] {};
			\vertex[fill] (r2) at (.5,-.5) [] {};
			\vertex[fill] (r3) at (1.5,.5) [] {};
			\vertex[fill] (r4) at (1.5,-.5) [] {};
			\vertex[fill] (r5) at (2,0) [] {};
			\path
			(r) edge (r1)
			(r) edge (r2)
			(r) edge (r3)
			(r) edge (r4)
			(r1) edge (r2)
			(r1) edge (r3)
			(r1) edge (r4)
			(r2) edge (r3)
			(r2) edge (r4)
			(r5) edge (r4)
			(r5) edge (r3)
			;
	\end{tikzpicture}}
	~~
	\subfloat[$D_2$]{\begin{tikzpicture}[scale=0.8]
			\vertex[fill] (r1) at (0,.8) [] {};
			\vertex[fill] (r2) at (.8,1) [] {};
			\vertex[fill] (r3) at (0,0) [] {};
			\vertex[fill] (r4) at (.8,-.2) [] {};
			\vertex[fill] (r5) at (1.6,.8) [] {};
			\vertex[fill] (r6) at (1.6,0) [] {};
			\vertex[fill] (r7) at (2.3,.4) [] {};
			\path
			(r5) edge (r2)
			(r1) edge (r2)
			(r1) edge (r5)
			(r1) edge (r3)
			(r1) edge (r4)
			(r2) edge (r3)
			(r2) edge (r4)
			(r3) edge (r6)
			(r4) edge (r3)
			(r4) edge (r6)
			(r5) edge (r6)
			(r7) edge (r6)
			(r5) edge (r7)
			;
	\end{tikzpicture}}
	~~
	\subfloat[$D_3$]{\begin{tikzpicture}[scale=0.8]
			\vertex[fill] (r1) at (0,0) [] {};
			\vertex[fill] (r2) at (0,.8) [] {};
			\vertex[fill] (r3) at (.8,1) [] {};
			\vertex[fill] (r4) at (.8,-.2) [] {};
			\vertex[fill] (r5) at (1.3,.4) [] {};
			\vertex[fill] (r6) at (2,.8) [] {};
			\vertex[fill] (r7) at (2,0) [] {};
			\vertex[fill] (r8) at (2.7,.4) [] {};
			\path
			(r5) edge (r2)
			(r1) edge (r2)
			(r1) edge (r5)
			(r1) edge (r3)
			(r1) edge (r4)
			(r2) edge (r3)
			(r2) edge (r4)
			(r3) edge (r6)
			(r4) edge (r3)
			(r4) edge (r7)
			(r5) edge (r7)
			(r5) edge (r6)
			(r6) edge (r7)
			(r7) edge (r8)
			(r6) edge (r8)  ;
	\end{tikzpicture}}
	~~
	\subfloat[$D_4$]{\begin{tikzpicture}[scale=0.7]
			\vertex[fill] (r1) at (0,1) [] {};
			\vertex[fill] (r2) at (.8,1.2) [] {};
			\vertex[fill] (r3) at (0,0) [] {};
			\vertex[fill] (r4) at (.8,-.2) [] {};
			\vertex[fill] (r5) at (1.6,1) [] {};
			\vertex[fill] (r6) at (1.6,0) [] {};
			\vertex[fill] (r7) at (2.6,1) [] {};
			\vertex[fill] (r8) at (2.6,0) [] {};
			\vertex[fill] (r9) at (3.1,.5) [] {};
			\path
			(r5) edge (r2)
			(r1) edge (r2)
			(r1) edge (r5)
			(r1) edge (r3)
			(r1) edge (r4)
			(r2) edge (r3)
			(r2) edge (r4)
			(r3) edge (r6)
			(r4) edge (r3)
			(r4) edge (r6)
			(r5) edge (r8)
			(r5) edge (r7)
			(r6) edge (r8)
			(r6) edge (r7)
			(r7) edge (r8)
			(r9) edge (r7)
			(r9) edge (r8)
			;
	\end{tikzpicture}}
	~~
	\subfloat[$D_5$]{\begin{tikzpicture}[scale=0.8]
			\vertex[fill] (r1) at (0,0) [] {};
			\vertex[fill] (r3) at (.5,.5) [] {};
			\vertex[fill] (r2) at (.5,-.5) [] {};
			\vertex[fill] (r5) at (1.5,.5) [] {};
			\vertex[fill] (r4) at (1.5,-.5) [] {};
			\vertex[fill] (r6) at (2.3,-.5) [] {};
			\vertex[fill] (r7) at (2.3,.5) [] {};
			\vertex[fill] (r8) at (3.3,-.5) []{};
			\vertex[fill] (r9) at (3.3,.5) [] {};
			\vertex[fill] (r10) at (3.8,0) [] {};
			\path
			(r1) edge (r2)
			(r1) edge (r3)
			(r1) edge (r4)
			(r1) edge (r5)
			(r2) edge (r3)
			(r2) edge (r4)
			(r2) edge (r5)
			(r3) edge (r4)
			(r3) edge (r5)
			(r6) edge (r4)
			(r5) edge (r7)
			(r6) edge (r7)
			(r6) edge (r8)
			(r6) edge (r9)
			(r8) edge (r7)
			(r7) edge (r9)
			(r8) edge (r9)
			(r8) edge (r10)
			(r9) edge (r10)
			;
	\end{tikzpicture}}
	\caption{The building blocks of the graphs  in $\Fq^*$}\label{fig:thm:quartic}
\end{figure} 

\begin{theorem}[Abdi and Ghorbani \cite{AbGh}] \label{thm:quartic}
	For any $n\ge11$, a graph with minimum spectral gap among the connected quartic graphs of order $n$ belongs to $\Fq^*$.
\end{theorem}

For given $n$, the family $\Fq^*$ may contain more than one graph of order $n$.
However, we \cite{AbGh} conjectured that the quartic graph with minimum spectral gap of any given order is unique with the structure described below. To this end, we  define a subfmaily $\Fq$ of $\Fq^*$ which contains a unique graph of any order $n\ge11$.
	Let $q$ and $r\le4$ be non-negative integers such that $n-11=5q+r$.
Then the $n$-vertex graph in $\Fq$ consists of $q$ middle blocks $M$ and each end block is determined by $r$ as follows. If $r=0$, then both end blocks are $D_1$. If $r=1$, then  the end blocks are $D_1$ and $D_2$. If $r=2$, then both end blocks are $D_2$. If $r=3$, then  the end blocks are $D_2$ and $D_3$.  Finally, if  $r=4$, then  the end blocks are $D_1$ and $D_5$.

\begin{conjecture}[\cite{AbGh}]\label{conj:MinQuartic}\rm
	For every $n\ge11$, the $n$-vertex graph  of $\Fq$
	is the unique graph with minimum spectral gap among connected quartic graphs of order $n$.	
\end{conjecture}

For every $n\equiv1\pmod5$, let us denote the graph of order $n$ of $\Fq$ by $\ga_n$.  The end blocks of $\ga_n$ are both $D_1$. Figure~\ref{fig:Gamma_n} illustrates this graph.

\begin{figure} 
	\centering
	\begin{tikzpicture}[scale=.9]
		\vertex[fill] (1) at (0,0) []{};
		\vertex[fill] (2) at (.5,.5) [] {};
		\vertex[fill] (3) at (.5,-.5) [] {};
		\vertex[fill] (4) at (1.5,.5) [] {};
		\vertex[fill] (5) at (1.5,-.5) [] {};
		\vertex[fill] (6) at (2,0) [] {};
		\vertex[fill] (7) at (2.5,.5) [] {};
		\vertex[fill] (8) at (2.5,-.5) [] {};
		\vertex[fill] (9) at (3.5,.5) [] {};
		\vertex[fill] (10) at (3.5,-.5) [] {};
		\vertex[fill] (11) at (4,0) [] {};
		\vertex[fill] (22) at (5.9,0) [] {};
		\vertex[fill] (23) at (6.4,.5) [] {};
		\vertex[fill] (24) at (6.4,-.5) [] {};
		\vertex[fill] (25) at (7.4,.5) [] {};
		\vertex[fill] (26) at (7.4,-.5) [] {};
		\vertex[fill] (27) at (8,0) [] {};
		\vertex[fill] (28) at (8.6,.5) [] {};
		\vertex[fill] (29) at (8.6,-.5) [] {};
		\vertex[fill] (30) at (9.6,.5) [] {};
		\vertex[fill] (31) at (9.6,-.5) [] {};
		\vertex[fill] (32) at (10.1,0) [] {};
		\tikzstyle{vertex}=[circle, draw, inner sep=0pt, minimum size=0pt]
		\vertex[] (12) at (4.2,.2)[] {};
		\vertex[] (13) at (4.2,-.2)[] {};
		\vertex[] (20) at (5.7,.2)[] {};
		\vertex[] (21) at (5.7,-.2)[] {};
		\tikzstyle{vertex}=[circle, draw, inner sep=0pt, minimum size=1pt]
		\vertex[fill] (14) at (4.75,0)[]{} ;
		\vertex[fill] (16) at (4.95,0)[]{} ;
		\vertex[fill] (18) at (5.15,0)[]{} ;
		\path
		(1) edge (2)
		(1) edge (3)
		(1) edge (4)
		(2) edge (3)
		(2) edge (4)
		(3) edge (4)
		(1) edge (5)
		(2) edge (5)	
		(3) edge (5)
		(4) edge (6)
		(5) edge (6)
		(6) edge (7)
		(6) edge (8)
		(7) edge (8)
		(7) edge (9)
		(7) edge (10)
		(8) edge (9)
		(8) edge (10)
		(9) edge (10)
		(9) edge (11)
		(10) edge (11)
		(11) edge (12)
		(11) edge (13)
		(20) edge (22)
		(21) edge (22)
		(22) edge (23)
		(22) edge (24)
		(23) edge (24)
		(23) edge (25)
		(23) edge (26)
		(24) edge (25)
		(24) edge (26)
		(25) edge (26)
		(25) edge (27)
		(26) edge (27)
		(27) edge (28)
		(27) edge (29)
		(28) edge (30)
		(28) edge (31)
		(28) edge (32)
		(29) edge (30)
		(29) edge (31)
		(29) edge (32)
		(30) edge (31)
		(30) edge (32)
		(31) edge (32)
		;
	\end{tikzpicture}
	\caption{The graph $\ga_n$}	
	\label{fig:Gamma_n}
\end{figure}
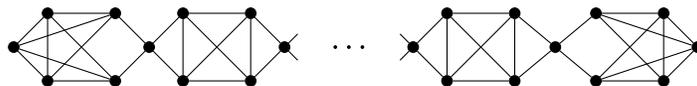

\subsection{Some basic facts on graph eigenvalues}

Let $G$ be a simple finite graph with vertex set $V(G)$ and edge set $E(G)$. Also let $A(G)$ be the adjacency matrix of $G$. It is well known that the smallest eigenvalue $\rho(G)$ satisfies the following:
\begin{equation*}\label{eq:least}
	\rho(G)=\min_{{\bf0}\ne\x\in\mathbb{R}^n}\frac{\x A(G)\x^\top}{\x\x^\top}.
\end{equation*}\label{eq:least1}
Also  if $\x=(x_1,\ldots,x_n)$, then
\begin{equation*}\label{eq2}
	\x A(G)\x^\top=2\sum_{ij\in E(G)}x_ix_j.
\end{equation*}
If $\x$ is an eigenvector corresponding to an eigenvalue $\lambda$ of $G$, then
\begin{equation}\label{eq:eigenvector}
	\lambda x_i=\sum_{j:\,ij\in E(G)}x_j, \quad\hbox{for}~ i=1, \ldots, n.
\end{equation}
We refer to \eqref{eq:eigenvector} as the {\em  eigen-equation}.

The following lemma gives the well-known eigenvalue
{\em interlacing} inequalities.

\begin{lemma}[{\rm \cite[p. $26$]{Haemers}}]\label{lem:interlacing}
	If $G$ is a graph  with eigenvalues $\lambda_1(G)\geq \cdots \geq \lambda_n(G)$, and
	$H$ is an induced subgraph of $G$ with eigenvalues $\lambda_1(H)\geq \cdots \geq \lambda_m(H)$, then
	\[\lambda_i(G)\geq \lambda_i(H) \geq \lambda_{n-m+i}(G), \quad\hbox{for}~i=1,\ldots, m.\]
\end{lemma}

A partition $\Pi= \lbrace C_1,  \ldots , C_{m} \rbrace$ of $V(G)$
is called an {\em equitable partition} for $G$ if for every pair of (not necessarily distinct) indices $i, j \in \lbrace 1,  \ldots , m\rbrace$, there is a non-negative integer $q_{ij}$ such that each vertex $v$ in the cell $C_i$ has exactly $q_{ij}$ neighbors in the cell $C_j$, regardless of the choice of $v$. The matrix  $Q = (q_{ij})$ is called the {\em quotient matrix} of $\Pi$. The following lemma gives a characterization of the eigenvectors of $G$ in terms of $\Pi$.

\begin{lemma} [{\rm \cite[p.~24]{Haemers}}] \label{lem:type}
	Let $G$ be a graph with an equitable partition  $\Pi$. The spectrum of $G$ consists of the spectrum of the quotient matrix of $\Pi$, with eigenvectors  that are constant on the cells of  $\Pi$, together with the eigenvalues belonging to eigenvectors whose components sum up to zero on each cell of $\Pi$.
\end{lemma}

This decomposition of the spectra has also a
group-theoretic meaning. It is taking the quotient with respect to
the irreducible representations of symmetry group (generated by flip
symmetries of the individual blocks of the graph); see \cite{band,mutlu}.
This interpretation, in particular, implies the self-adjointness of the quotient matrix which will be used in our arguments.

\section{Infinite periodic graphs and Floque--Bloch theory}\label{sec:Floque-Bloch}

Removing the end blocks of the graphs $\de_n$ and $\ga_n$, the resulting graphs have a repetitive structure which are indeed finite chunks of two infinite periodic
graphs. We denote these infinite graphs by $\bde$ and $\bga$, respectively.
The spectrum (of the adjacency operator; see below) of such periodic infinite graphs can be investigated by some tools from
mathematical physics, namely the Floquet--Bloch transform. This provides a powerful tool to investigate the band structure (and the gaps in the spectra) of
the infinite periodic graphs. For an introduction see  \cite[Chapter 4]{Berkolaiko} and \cite{Exner}.
Koll\'ar and Sarnak \cite{Sarnak}  also utilized this to construct cubic graphs with `extremal' gaps.

In what follows, we give a brief account of Floquet--Bloch theory and apply it to $\bde$ as the running example.  Further details can be found in \cite{Berkolaiko,Exner,Harrison}.

 The adjacency matrix $A$ of a finite graph $G$ can be considered as an operator  acting on the functions $f:V(G)\to\Rbb$ or $\mathbb{C}$ defined by
\begin{equation}\label{eq:Af}
(Af)(u)=\sum_{v:\,uv\in E(G)}f(v),\quad\hbox{for all $u\in V(G)$.}
\end{equation}
Let $\bG$ be a (locally finite) infinite graph with countable number of vertices. 
As an analogue to  the adjacency matrix in the finite case, one can define the {\em adjacency operator}  $A(\bG)$ acting
on the functions  $f:V(\bG)\to\mathbb{C}$ defined in the same way as in \eqref{eq:Af}.
In order to be able to have the notion of `spectrum' for  $A(\bG)$, we restrict its domain to the space
 $\ell^2(\bG)$, that is the Hilbert space of square-summable complex functions defined on $V(\bG)$ (see, for instance \cite[Appendix~C]{Berkolaiko}).

An infinite graph is called $\mathbb{Z}$-{\em periodic} if it is equipped with an action of the free Abelian group $\mathbb{Z}$. Roughly speaking, we say that a graph is $\mathbb{Z}$-periodic if it is built of an infinite number of copies of a fixed compact  graph, the {\em fundamental domain}, glued together along a two-way infinite path. The choice of a fundamental domain is not unique  (see \cite[p.~106]{Berkolaiko} for more details). Figure~\ref{fig:W} illustrates a fundamental domain for each of the graphs $\bde$ and $\bga$. 

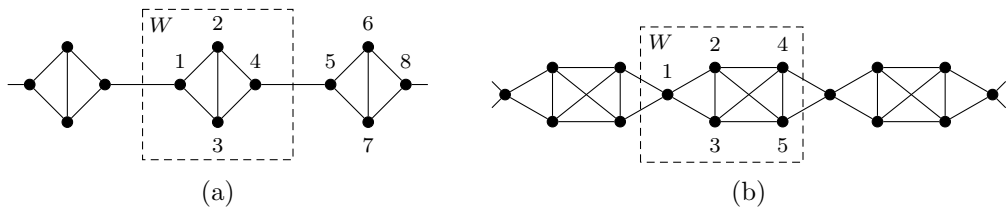
\begin{figure}[h!]
	\centering
	\subfloat[]{\begin{tikzpicture}[scale=1]
			\vertex[fill] (2) at (0,0) [] {};
			\vertex[fill] (3) at (.5,-.5) [] {};
			\vertex[fill] (4) at (.5,.5) [] {};
			\vertex[fill] (5) at (1,0) [] {};
			\vertex[fill] (6) at (2,0) [label=above:\scriptsize{$1$}] {};
			\vertex[fill] (7) at (2.5,.5) [label=above:\scriptsize{$2$}] {};
			\vertex[fill] (8) at (2.5,-.5) [label=below:\scriptsize{$3$}] {};
			\vertex[fill] (9) at (3,0) [label=above:\scriptsize{$4$}] {};
			\vertex[fill] (10) at (4,0) [label=above:\scriptsize{$5$}] {};
			\vertex[fill] (11) at (4.5,.5) [label=above:\scriptsize{$6$}] {};
			\vertex[fill] (12) at (4.5,-.5)  [label=below:\scriptsize{$7$}] {};
			\vertex[fill] (13) at (5,0)  [label=above:\scriptsize{$8$}] {};
			\tikzstyle{vertex}=[circle, draw, inner sep=0pt, minimum size=0pt]
			\vertex (s) at (-.3,0)[] {};
			\vertex (ss) at (5.3,0) [] {};
			\vertex[] () at (1.9,.55) [label=above:\scriptsize{$W\ \ \ $}] {};
			\path
			(4) edge (3)
			(2) edge (3)
			(2) edge (4)
			(3) edge (5)
			(4) edge (5)
			(5) edge (6)
			(6) edge (7)
			(6) edge (8)
			(7) edge (8)
			(7) edge (9)
			(8) edge (9)
			(9) edge (10)
			(10) edge (11)
			(10) edge (12)
			(11) edge (12)
			(11) edge (13)
			(12) edge (13)
			(2) edge (s)
			(ss) edge (13)  ;
			\draw[densely dashed]  (1.5,1) rectangle (3.5,-1) ;
		\end{tikzpicture}\label{fig:Wcubic}}
	\quad \quad
	\subfloat[]{\begin{tikzpicture}[scale=0.9]
			\vertex[fill] (r13) at (5.1,.4) [] {};
			\vertex[fill] (r14) at (5.8,.8) [] {};
			\vertex[fill] (r15) at (5.8,0) [] {};
			\vertex[fill] (r16) at (6.8,.8) [] {};
			\vertex[fill] (r17) at (6.8,0) [] {};
			\vertex[fill] (r18) at (7.5,.4) [label=above:\scriptsize{$1$}] {};
			\vertex[fill] (r19) at (8.2,.8) [label=above:\scriptsize{$2$}] {};
			\vertex[fill] (r20) at (8.2,0) [label=below:\scriptsize{$3$}] {};
			\vertex[fill] (r21) at (9.2,.8) [label=above:\scriptsize{$4$}] {};
			\vertex[fill] (r22) at (9.2,0) [label=below:\scriptsize{$5$}] {};
			\vertex[fill] (r23) at (9.9,.4)  {};
			\vertex[fill] (r24) at (10.6,.8) [] {};
			\vertex[fill] (r25) at (10.6,0) [] {};
			\vertex[fill] (r26) at (11.6,.8) [] {};
			\vertex[fill] (r27) at (11.6,0) [] {};
			\vertex[fill] (r28) at (12.3,.4) [] {};
			\tikzstyle{vertex}=[circle, draw, inner sep=0pt, minimum size=0pt]
			\vertex (s) at (12.5,.6)[] {};
			\vertex (ss) at (12.5,.2) [] {};
			\vertex (sss) at (4.9,.6)[] {};
			\vertex (ssss) at (4.9,.2) [] {};
			\vertex[] () at (7.4,.9) [label=above:\scriptsize{$W$}] {};
			\path
			(r13) edge (r14)
			(r13) edge (r15)
			(r14) edge (r15)
			(r14) edge (r16)
			(r14) edge (r17)
			(r15) edge (r16)
			(r15) edge (r17)
			(r16) edge (r17)
			(r18) edge (r17)
			(r16) edge (r18)
			(r18) edge (r19)
			(r18) edge (r20)
			(r19) edge (r20)
			(r19) edge (r21)
			(r19) edge (r22)
			(r20) edge (r21)
			(r20) edge (r22)
			(r21) edge (r23)
			(r21) edge (r22)
			(r22) edge (r23)
			(r23) edge (r24)
			(r23) edge (r25)
			(r24) edge (r25)
			(r24) edge (r26)
			(r24) edge (r27)
			(r25) edge (r26)
			(r25) edge (r27)
			(r26) edge (r27)
			(r26) edge (r28)
			(r27) edge (r28)
			(s) edge (r28)
			(ss) edge (r28)
			(sss) edge (r13)
			(ssss) edge (r13) ;
			\draw[densely dashed]  (7.1,-.6) rectangle (9.5,1.4) ;
		\end{tikzpicture}\label{fig:Wquartic}}
	\caption{Fundamental domains $W$ for the infinite periodic graphs $\bde$ and $\bga$}\label{fig:W}
\end{figure} 

For $q\in\Z$, let $T_q$ denote the $q$-shift automorphism of $\bde$. For instance, considering the labels of Figure~\ref{fig:Wcubic}, $T_1$ shifts the vertices $1,2,3,4$ to the vertices $5,6,7,8$, respectively.
We can also say that the neighbors of the vertex $1$ are the vertices $2$, $3$ and $T_{-1}(4)$.

For a given $\theta\in [-\pi,\pi]$, let $\ell^2_\theta$ denote the subspace of $\ell^2(\bG)$ consisting of those functions $f$ such that
$$f(T_q(v))=e^{iq\theta}f(v),\quad\hbox{for all}~q\in\Z, v\in V(\bde).$$
Such a function $f$ is clearly uniquely determined
by the vector $\left(f(1) ,f(2),f(3),f(4)\right)$ of its values at the four vertices in $W$, and thus $\ell^2_\theta$
is $4$-dimensional and naturally isomorphic to $\ell^2(W)$.
We now define the operator $A_\theta : \ell^2(W)\to\ell^2(W)$ as the restriction to
the space $\ell^2_\theta$ of the operator $A(\bde)$.
More precisely, for $f\in\ell^2_\theta$, we have
\begin{align*}
(A_\theta f)(1)&=f(2)+f(3)+f(T_{-1}(4))=f(2)+f(3)+e^{-i\theta}f(4),\\
(A_\theta f)(2)&=f(1)+f(3)+f(4),\\
(A_\theta f)(3)&=f(1)+f(2)+f(4),\\
(A_\theta f)(4)&=f(T_{1}(1))+f(2)+f(3)=e^{i\theta}f(1)+f(2)+f(3).
\end{align*}
It follows that $A_\theta$ can be identified by the following matrix:
$$\begin{pmatrix}
 0&1&1&e^{i\theta}\\
 1&0&1&1\\
 1&1&0&1\\
 e^{-i\theta}&1&1&0
 \end{pmatrix}.$$
Let $\la_1(\theta)\ge\cdots\ge\la_4(\theta)$ be the eigenvalues of $A_\theta$.

Floquet--Bloch theory says that the spectrum of
$A(\bG)$ is equal to the union of the spectra of the operators
$A_\theta$ for all  $\theta\in [-\pi,\pi]$.
\begin{theorem}[see {\cite[Theorem~4.3.1]{Berkolaiko}}]\label{thm:Floquet-Bloch}
 The spectrum of $A(\bde)$ is given by
$\textstyle\bigcup\limits_{\theta\in [-\pi,\pi]}\{\la_1(\theta),\ldots,\la_4(\theta)\}$.
Moreover each of the functions $\theta\mapsto\la_j(\theta)$ is a  continuous function.
\end{theorem}
 The segments $$I_j =\textstyle\bigcup\limits_{\theta\in [-\pi,\pi]}\la_j(\theta)$$ are called {\em bands of
the spectrum of} $A(\bde)$.

By straightforward computations, we obtain
\begin{align*}
\la_1(\theta)&=\tfrac13-\tfrac83\sin\left(h(\theta)\right),\\
\la_2(\theta)&=\tfrac13-\tfrac83\cos\left(h(\theta)+\pi/6\right),\\
\la_3(\theta)&=-1,\\
\la_4(\theta)&=\tfrac13+\tfrac83\sin\left(h(\theta)+\pi/3\right),
\end{align*}
where $h(\theta)=\frac13\sin^{-1}\left(\frac{5+27\cos(\theta)}{32}\right)$. In view of Theorem~\ref{thm:Floquet-Bloch}, these four functions determine the spectrum of $A(\bde)$.  They are visualized in Figure~\ref{fig:plotCubic}. It is seen that bands for $A(\bde)$ are $[-\sqrt5,-1]$, $[-1,1]$ and $[\sqrt5,3]$. Also, $(1,\sqrt{5})$ is a  gap interval for $\bde$. 

For the graph $\bga$, the corresponding matrix $A_\theta$ for the fundamental domain given in  Figure~\ref{fig:Wquartic} is as follows:
$$\begin{pmatrix}
 0&1&1&e^{i\theta}&e^{i\theta}\\
 1&0&1&1&1\\
 1&1&0&1&1\\
 e^{-i\theta}&1&1&0&1\\
 e^{-i\theta}&1&1&1&0
 \end{pmatrix}.$$
Similarly to the spectrum of $\bde$, the functions identifying the spectrum of $A(\bga)$ can be computed. We illustrate them
 in Figure~\ref{fig:plotQuartic} (note that here $-1$ has multiplicity two). It is seen that $[(-1-\sqrt{17})/2,-1]$, $[-1,(-1+\sqrt{17})/2]$ and $[3,4]$ are the bands and $\left((-1+\sqrt{17})/2,3\right)$ is a  gap interval.

 \input{fig-plotCubicQuartic.txt}
The spectra of the finite chunks will fill up the bands of the
infinite graph. But the boundary can create persistent eigenvalues in
the gaps. This is a highly non-trivial effect and has been linked
with the Chern number of the eigenvector bundle over the Brillouin
zone (which is a circle in this 1-dimensional setup). This is known
as the ``bulk-boundary correspondence" in the study of topological
insulators. See the book \cite{aopBook} for a very approachable introduction.
Proving that the boundary makes the eigenvalues appear (or not) in
the gap is in general a rather difficult job.

To get an idea about what the situation is with our graphs  $\de_n$ and $\ga_n$, we have computed their spectrum of  for several values   $n$; see the diagrams given in Figures~\ref{fig:EVC} and \ref{fig:EVQ}. The diagrams suggest that  $\de_n$ and $\ga_n$ retain the gaps $(1,\sqrt{5})$ and $\left((-1+\sqrt{17})/2,3\right)$, respectively. This is indeed our ultimate goal to be settled in the rest of the paper. 
 
\input{fig-EVQ.txt}

\section{Simplicity of eigenvalues of minimal regular graphs}\label{sec:simplicity}

In this section, we show that, except for two eigenvalues, namely $0,-1$, any eigenvalue of the graphs in $\Fc$ is simple. We also prove that, except for at most eight eigenvalues, the eigenvalues of the  graphs in $\Fq^*$ are simple.
Moreover, the number of exceptional eigenvalues can be reduced to four  if Conjecture~\ref{conj:MinQuartic} is true.

\subsection{Minimal cubic graphs}

Let $G\in\Fc$.
Every graph in $G$  possesses an equitable partition in which each cell has size $1$ or $2$, consisting of the vertices drawn vertically above each other as  displayed in  Figure~\ref{fig:mincubic}. We denote this equitable partition by $\Pi$ in this subsection.
Let $\lambda$ be an eigenvalue of $G$ with eigenvector $\x$. In view of Lemma~\ref{lem:type}, $\x$ is either constant on the cells of $\Pi$ or it is orthogonal to each cell of $\Pi$.
Accordingly, we call $\x$ of the {\em first} or {\em second type}, respectively. We first show that eigenvalues with second type eigenvectors are very restricted.
\begin{lemma}\label{lem:2ndtypeFc}
 Let $G\in\Fc$ and $\la$ be an eigenvalue of $G$ with an eigenvector of second type. If  $n\equiv2\pmod4$, then 
 $\lambda\in\{ 0, -1\}$ and if  $n\equiv0\pmod4$, then  
  $\lambda\in\left\{ 0, -1 , (-1\pm\sqrt5)/2\right\}$.
\end{lemma}
\begin{proof}
Let $\x$ be an eigenvector of $\la$ of the second type. 
 The components of $\x$ sum up to zero on each cell of the partition $\Pi$. In particular, the components of $\x$ on the cut vertices are zero.
Note that $\x\neq 0$ and we can always find a non-$K_2$ block $B$ such that  the components of $\x$ are not all zero on it.
Firstly, let $B$ be a middle block. The components of $\x$ on $B$ are as shown in Figure~\ref{fig:thcubica}.
Since the vertices $i$ and ${i+1}$ belong to a cell and $\x$ is of the second type, we have $x_i+x_{i+1}=0$.
Also by the eigen-equation, we see that $\lambda  x_i=x_{i+1}$. It follows that  $\lambda=-1$.
Next, let $B$ be an end block. The components of $\x$ on $B$ are as shown in Figure~\ref{fig:thcubicb} or \ref{fig:thcubicc}.
In the first case, if $x_3\neq 0$, then $\lambda x_3=0$, so $\lambda=0$.
Otherwise, $x_1\neq 0$ and from $\lambda x_1=-x_1$, we obtain that $\lambda=-1$.
In the second case, we have  $\lambda x_3=x_5$. So if $x_5\neq 0$, then $x_3\neq 0$. Hence from $\lambda x_5=-x_5+x_3$ and $\lambda x_3=x_5$, we obtain $\lambda^2+\lambda -1=0$, that is $\lambda=(-1\pm\sqrt{5})/{2}$.
If $x_5=0$, then $x_3=\lambda x_5=0$.  Therefore,
$x_1\neq 0$ and then by $\lambda x_1=-x_1$, we have $\lambda=-1$.
\end{proof}
\begin{figure}[h!]\label{fig:blocks}
 \centering
 \subfloat[]{\begin{tikzpicture}[scale=1]
    \vertex[fill] (6) at (2,0) [label=left:\scriptsize{$0$}] {};
    \vertex[fill] (7) at (2.5,.5) [label=above:\scriptsize{$x_i$}] {};
    \vertex[fill] (8) at (2.5,-.5) [label=below:\scriptsize{$x_{i+1}=-x_i$}] {};
    \vertex[fill] (9) at (3,0) [label=right:\scriptsize{$0$}] {};
	\path
	    (6) edge (7)
	    (6) edge (8)
	    (7) edge (8)
	    (7) edge (9)
	    (8) edge (9);
\end{tikzpicture}\label{fig:thcubica}}
\quad \quad
\subfloat[]{\begin{tikzpicture}[scale=1]
    \vertex[fill] (34) at (-11.7,-.5) [label=below:\scriptsize{$-x_1$}] {};
	\vertex[fill] (33) at (-11.7,.5) [label=above:\scriptsize{$x_1$}] {};
	\vertex[fill] (32) at (-10.7,-.5) [label=below:\scriptsize{$-x_3$}] {};
	\vertex[fill] (31) at (-10.7,.5) [label=above:\scriptsize{$x_3$}] {};
    \vertex[fill] (30) at (-10.2,0) [label=right:\scriptsize{$0$}] {};
	\path
	    (30) edge (31)
	    (30) edge (32)
	    (31) edge (33)
	    (31) edge (34)
	    (32) edge (33)
	    (32) edge (34)
	    (33) edge (34);
\end{tikzpicture}\label{fig:thcubicb}}
\quad \quad
\subfloat[]{\begin{tikzpicture}[scale=0.9]
    \vertex[fill] (34) at (-10.9,-.5) [label=below:\scriptsize{$-x_3$}] {};
	\vertex[fill] (33) at (-10.9,.5) [label=above:\scriptsize{$x_3$}] {};
	\vertex[fill] (32) at (-9.9,-.5) [label=below:\scriptsize{$-x_5$}] {};
	\vertex[fill] (31) at (-9.9,.5) [label=above:\scriptsize{$x_5$}] {};
    \vertex[fill] (30) at (-9.4,0) [label=right:\scriptsize{$0$}] {};
   \vertex[fill] (35) at (-11.9,-.5) [label=below:\scriptsize{$-x_1$}] {};
	\vertex[fill] (36) at (-11.9,.5) [label=above:\scriptsize{$x_1$}] {};
	\path
	    (30) edge (31)
	    (30) edge (32)
        (31) edge (32)
	    (31) edge (33)
	    (32) edge (34)
	    (33) edge (35)
	    (34) edge (36)
	    (33) edge (36)
	    (34) edge (35)
	    (35) edge (36);
\end{tikzpicture}\label{fig:thcubicc}}
\caption{ The  blocks of a minimal cubic graph and the components of an eigenvector of the second type }
\end{figure}
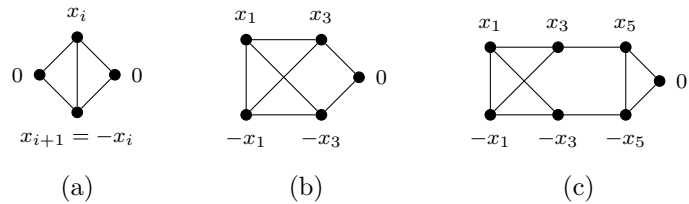

 For  $G\in\Fc$, the components of a eigenvector $\x$ of the first type are equal on the vertices in each cell. We call the components of $\x$ on the cells of size $2$ {\em main components} of $\x$; from the last block, we include only one such a cell, see Figure~\ref{fig:a}. 
 We will see that all other components of $\x$ are determined by its main components.  
 For $\la\ne0$, we can assume that the first main component is $1$; otherwise it must be $0$, which in turn by recursive application of the eigen-equation implies that $\x=\bf0$, a contradiction.
\begin{figure}
 \centering
\begin{tikzpicture}[scale=0.9]
	\vertex[fill] (1) at (0,-.5) [] {};
	\vertex[fill] (2) at (0,.5) [label=above:\footnotesize{$1$}] {};
	\vertex[fill] (3) at (1,-.5) [] {};
	\vertex[fill] (4) at (1,.5) [label=above:\footnotesize{$a_0$}] {};
    \vertex[fill] (5) at (1.5,0) [] {};
    \vertex[fill] (6) at (2,0) [] {};
    \vertex[fill] (7) at (2.5,.5) [label=above:\footnotesize{$a_1$}] {};
    \vertex[fill] (8) at (2.5,-.5) [] {};
    \vertex[fill] (9) at (3,0) [] {};
    \vertex[fill] (10) at (3.5,0) [] {};
    \vertex[fill] (11) at (4,.5) [label=above:\footnotesize{$a_2$}] {};
    \vertex[fill] (12) at (4,-.5) [] {};
    \vertex[fill] (13) at (4.5,0) [] {};
    \vertex[fill] (22) at (6.4,0) [] {};
    \vertex[fill] (23) at (6.9,.5) [label=above:\footnotesize{$a_{m-1}$}] {};
    \vertex[fill] (24) at (6.9,-.5) [] {};
    \vertex[fill] (25) at (7.4,0) [] {};
    \vertex[fill] (26) at (7.9,0) [] {};
    \vertex[fill] (27) at (8.4,.5) [label=above:\footnotesize{$a_{m}$}] {};
    \vertex[fill] (28) at (8.4,-.5) [] {};
    \vertex[fill] (29) at (8.9,0) [] {};
	\vertex[fill] (32) at (9.9,-.5) [] {};
	\vertex[fill] (31) at (9.9,.5) [label=above:\footnotesize{$a_{m+1}\ \ \ $}] {};
    \vertex[fill] (30) at (9.4,0) [] {};
    \tikzstyle{vertex}=[circle, draw, inner sep=0pt, minimum size=1pt]
    \vertex[fill] (15) at (5.25,0) [] {};
    \vertex[fill] (17) at (5.45,0) [] {};
    \vertex[fill] (19) at (5.65,0) [] {};
    \tikzstyle{vertex}=[circle, draw, inner sep=0pt, minimum size=0pt]
    \vertex[] (14) at (4.8,0) [] {};
    \vertex[] (21) at (6.1,0) [] {};
    \vertex[] (33) at (10.4,.5) [] {};
    \vertex[] (34) at (10.4,-.5) [] {};
    \vertex[] (35) at (10.4,.3) [] {};
    \vertex[] (36) at (10.4,-.3) [] {};
	\path
		(1) edge (2)
		(1) edge (3)
	    (1) edge (4)
		(2) edge (3)
        (2) edge (4)
		(3) edge (5)
	    (4) edge (5)
	    (5) edge (6)
	    (6) edge (7)
	    (6) edge (8)
	    (7) edge (8)
	    (7) edge (9)
	    (8) edge (9)
	    (9) edge (10)
	    (10) edge (11)
	    (10) edge (12)
	    (11) edge (12)
	    (11) edge (13)
	    (12) edge (13)
	    (13) edge (14)
	    (21) edge (22)
	    (22) edge (23)
	    (22) edge (24)
	    (23) edge (24)
	    (23) edge (25)
	    (24) edge (25)
	    (25) edge (26)
	    (26) edge (27)
	    (26) edge (28)
	    (27) edge (28)
	    (27) edge (29)
	    (28) edge (29)
	     (29) edge (30)
	     (31) edge (30)
	     (32) edge (30)
	     (31) edge (33)
	     (32) edge (34);
	      \path[densely dashed]
	     (31) edge (35)
	     (36) edge (32)
	     (31) edge (32);
\end{tikzpicture}
\caption{Main components of an eigenvector of first type for  $G\in \Fc$}
\label{fig:a}
\end{figure}
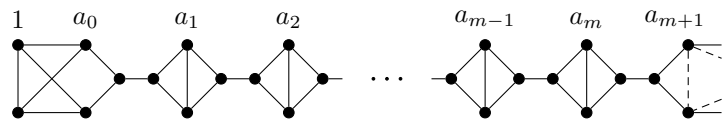 

The repetitive structure of the graphs in $\Fc$ induces a recurrence relation on the components of  eigenvectors of first type. This allows us to derive an elegant description of the components in the following lemma. The lemma provides a key tool to decide whether an eigenvalue can belong to a specific set.

\begin{lemma}\label{lem:Recursion}
Let $G\in \Fc$ and $\lambda\not\in\{0,\pm1,\pm\sqrt5,3\}$ be an eigenvalue of $G$ with an eigenvector $\x$ of the first type. Then the main components of $\x$ are given by
\begin{equation}\label{eq:PQRS}
 a_j=\frac{\lambda-1}{2^{j+2}}\left((P+1)\big(k-\sqrt{k^2-4}\big)^j-(P-1)\big(k+\sqrt{k^2-4}\big)^j\right),~ j=0,1,\ldots,m+1,
\end{equation}
where
\begin{equation}\label{eq:P}
P=\frac{-\la^3+k+7\la+4}{\sqrt{k^2-4}}\quad\hbox{and}\quad k=\frac12(\lambda^3-\lambda^2-5\lambda+1).
\end{equation}
\end{lemma}

\begin{proof}
 Consider three consecutive middle blocks of $G$, as depicted in Figure~\ref{fig:3middle},  in which the  labels of the vertices indicate the components of $\x$.
   \begin{figure}[h!]
 \centering
\begin{tikzpicture}[scale=1]
	\vertex[fill] (2) at (0,0) [] {};
	\vertex[fill] (3) at (.5,-.5) [] {};
	\vertex[fill] (4) at (.5,.5) [label=above:\scriptsize{$a_{i-2}$}] {};
    \vertex[fill] (5) at (1,0) [label=above:\scriptsize{$\ \ b_{i-2}$}] {};
    \vertex[fill] (6) at (2,0) [label=above:\scriptsize{$c_{i-2}\ \ $}] {};
    \vertex[fill] (7) at (2.5,.5) [label=above:\scriptsize{$a_{i-1}$}] {};
    \vertex[fill] (8) at (2.5,-.5) [] {};
    \vertex[fill] (9) at (3,0) [label=above:\scriptsize{$\ \ b_{i-1}$}] {};
    \vertex[fill] (10) at (4,0) [label=above:\scriptsize{$c_{i-1}\ \ $}] {};
    \vertex[fill] (11) at (4.5,.5) [label=above:\scriptsize{$a_{i}$}] {};
    \vertex[fill] (12) at (4.5,-.5) [] {};
    \vertex[fill] (13) at (5,0) [] {};
	\path
		(4) edge (3)
		(2) edge (3)
        (2) edge (4)
		(3) edge (5)
	    (4) edge (5)
	    (5) edge (6)
	    (6) edge (7)
	    (6) edge (8)
	    (7) edge (8)
	    (7) edge (9)
	    (8) edge (9)
	    (9) edge (10)
	    (10) edge (11)
	    (10) edge (12)
	    (11) edge (12)
	    (11) edge (13)
	    (12) edge (13);
\end{tikzpicture}
\caption{Three middle blocks of  $G\in \Fc$ and  the components of $\x$} \label{fig:3middle}
\end{figure}

 Using the eigen-equation, we obtain
$$\begin{array}{rrr}
\lambda b_{i-2}-c_{i-2}-2a_{i-2}=0,& \lambda c_{i-2}-b_{i-2}-2a_{i-1}=0,& \lambda a_{i-1}-c_{i-2}-a_{i-1}-b_{i-1}=0, \\
\lambda b_{i-1}-2a_{i-1}-c_{i-1}=0,& \lambda  c_{i-1}-b_{i-1}-2a_i=0.&
\end{array}$$
 From these equations we can  write $a_{i-1}$ in terms of $a_{i-2}$ and  $a_{i}$ as:
$$a_{i-1}=\frac{2(a_{i-2}+a_{i})}{\lambda^3-\lambda^2-5\lambda+1}.$$
So we come up with the following recurrence relation:
\begin{equation}\label{eq:Recursion}
	\begin{array}{l}
		a_i=k a_{i-1}-a_{i-2},\quad i=2,\ldots,m+1, \\
		a_0=\frac{1}{2}(\lambda-1),~ a_1=\frac{1}{4}(\lambda^4-\lambda^3-7\lambda^2+3\lambda+4),
	\end{array}
\end{equation}
where the initial conditions are obtained using the eigen-equation on the first four cells of $G$.
To solve the recurrence relation \eqref{eq:Recursion}, we find the zeros of its characteristic equation
 $x^2-kx+1=0$, that is $$R=(k-\sqrt{k^2-4})/2,\quad S=(k+\sqrt{k^2-4})/2.$$
Since $\lambda\not\in\{\pm1,\pm\sqrt5,3\}$, we have  $k^2-4=(\lambda-1)(\la-3)
		(\la^2-5)(\la+1)^2\ne0$. So $R,S$ are distict (reals or complex numbers). 
 It turns out that the solution of \eqref{eq:Recursion} is
\begin{equation}\label{eq:PQRS2}
 a_j=\frac{\lambda-1}{4}\left((P+1)R^j-(P-1)S^j\right),
\end{equation}
form which the result follows.
\end{proof}

The relation obtained in Lemma~\ref{lem:Recursion} for the main components of eigenvectors of first type and a similar relation for the minimal quartic graphs (which will be settled later) play crucial roles in the proof of our main results. 

We are now prepared to prove that eigenvalues of minimal cubic graphs  except possibly for $0,-1$ are simple.
\begin{theorem}\label{thm:simplecubic}
 Let $G\in \Fc$ be of order $n\ge10$.
If  $n\equiv2\pmod4$, then all the eigenvalues of $G$ except $-1$ and $0$ are simple. If $n\equiv 0\pmod4$, then the only non-simple eigenvalue of $G$ is $-1$.
\end{theorem}
\begin{proof}
 Let $Q$ be the quotient matrix of $G$ (with respect to $\Pi$).
We first show that any  eigenvalue $\la$  of $Q$ is  simple. One can easily see that  due to the path-like structure of $G$, $Q$ is a  tridiagonal matrix with non-zero subdiagonal entries. 
The matrix $Q$ is self-adjoint as an operator in the weighted $L^2$ space; the correct weight is the size of the $j$-th cell.
Thus $Q$  is matrix-similar to
a symmetric matrix and so it is a  diagonalizable tridiagonal matrix with non-zero subdiagonal entries.
The submatrix obtained from deleting the first row and the last column of $Q-\lambda I$ is
an upper triangonal matrix whose diagonal entries are all non-zero and so it is non-singular. Therefore, the
nullity of $Q-\lambda I$ is $1$. It follows then that the geometric multiplicity of $\la$ is $1$ and because $Q$ is diagonalizable, the algebraic multiplicity of $\lambda$ is $1$ as well. Therefore, $\lambda$ is a simple eigenvalue of $Q$.

For $n\equiv2\pmod4$, by Lemma~\ref{lem:2ndtypeFc}, the only eigenvalues with second type eigenvectors are $0$ and $-1$. So if $\la\ne0,-1$, then $\la$ is an eigenvalues of $Q$ 
and by the above argument has multiplicity 1. 

Next, assume that $n\equiv 0\pmod4$. We show that neither $0$ nor $(-1\pm\sqrt5)/2$ can be an eigenvalue of $G$ with eigenvectors  of the first type (and so not an eigenvalue of $Q$). First consider $(-1\pm\sqrt5)/2$. It is enough to argue only for $(-1+\sqrt5)/2$ (since if an algebraic integer is an eigenvalue of $Q$, its {\em algebraic conjugate} is also an eigenvalue of $Q$).
Now, for a contradiction, let  $(-1+\sqrt5)/2$  be an  eigenvalue  with eigenvector  $\x$ of the first type.
By Lemma~\ref{lem:Recursion}, the main components of $\x$  satisfies \eqref{eq:PQRS} or equivalently  \eqref{eq:PQRS2}.
On the other hand, by applying the eigen-equation on the last eight cells of $G$, we observe that $a_{m-1}=a_m=x_n$.
With the notation of \eqref{eq:PQRS2} for $\la=(-1+\sqrt5)/2$, we have
$$k=-\frac{\sqrt5}2,~P=-\frac{i}{\sqrt{11}}\left(4\sqrt{5}+5\right), ~R= -\frac{\sqrt5+i\sqrt{11}}4.$$
Note that here $S=\bar R$, the complex conjugate of $R$.
Let $q:=(1+P)/|1+P|$. Then combining \eqref{eq:PQRS2} and $a_{m-1}=a_m$ results in
$$qR^m+\bar{q}\bar{R}^m=qR^{m-1}+\bar{q}\bar{R}^{m-1},$$
which implies that ${R}^{2m-1}=\bar q^2$. 
 It is seen that this cannot hold, for example by transforming into polar coordinates, which gives a contradiction.
 
We can argue similarly for $\la=0$ where by applying  Lemma~\ref{lem:Recursion}, we obtain
$$a_j=\frac{-5+3i\sqrt{15}}{20}\left(\frac{1-i\sqrt{15}}{4}\right)^j-\frac{5+3i\sqrt{15}}{20}\left(\frac{1+i\sqrt{15}}{4}\right)^j,
\quad j=1, \ldots, m.$$
On the other hand by applying the eigen-equation on the right end block, we see that $a_m=-\frac{10}{11}a_{m-1}$. Straightforward computation shows that this leads to a contradiction.
 Therefore,  $0$ and $(-1\pm\sqrt5)/2$  cannot be eigenvalues of $Q$.

Now, if $\lambda\notin\left\{ 0, -1 , (-1\pm\sqrt5)/2\right\}$, then $\x$ is of the first type. So the multiplicity of $\la$ for $G$ is the same as its multiplicity for $Q$.
 Thus, $\la$ is a simple eigenvalue of $G$. The proof of Lemma~\ref{lem:2ndtypeFc} shows that
the eigenvalues $(-1\pm\sqrt{5})/2$  have eigenvectors of second type which come from the end block of Figure~\ref{fig:thcubicc}.
Since there is at most one such end block in $G$, the multiplicity of $(-1\pm\sqrt{5})/2$
is at most $1$. Again  by the proof of Lemma~\ref{lem:2ndtypeFc}, the eigenvalue $0$  has to have a second type eigenvector coming from the end block of Figure~\ref{fig:thcubicb}. Thus when $n\equiv0 \pmod4$, $0$ is a simple eigenvalue, completing the proof.
\end{proof}

\subsection{Minimal quartic graphs}

Similarly to the minimal cubic graphs, members of $\Fq^*$ have an equitable partition with cells of size $1$ or $2$ consisting of the vertices
drawn vertically above each other in the building blocks of Figure~\ref{fig:thm:quartic} (with the exceptions that the first cell in the blocks $D_1$ and $D_5$ has three vertices and the  first cell in the blocks $D_2$ and $D_4$ has four vertices). We denote this equitable partition by $\Pi$ in this subsection. 	Let $\lambda$ be an eigenvalue of $G$ with eigenvector $\x$. In view of Lemma~\ref{lem:type}, $\x$ is either constant on the cells of $\Pi$ or $\x$ is orthogonal to each cell of $\Pi$.
Accordingly, we call $\x$ of the {\em first} or {\em second type}, respectively.

\begin{lemma}\label{lem:1stTypeQua}
For any graphs in $\Fq^*$, its quotient matrix has only simple eigenvalues.
\end{lemma}

\begin{proof}
Let $G\in\Fq^*$ and $\la$ be an eigenvalue for the  quotient matrix $Q$ of $G$ (with respect to $\Pi$). So $\la$ gives rise to an eigenvalue of $G$ with an eigenvector $\x$ of the first type. 
First, suppose that $G$ does not contain $D_3$ or its mirror image.
In this case, $Q$ is a  tridiagonal matrix with non-zero subdiagonal entries.
Further, $Q$ is diagonalizable, and so similarly to the proof of Theorem~\ref{thm:simplecubic}, all the eigenvalues of $Q$ are simple. Next, assume that one or both  end blocks of $G$ are $D_3$.
In this case, $Q$ is  not  tridiagonal anymore. However, we show that the nullity of $Q-\lambda I$ is $1$.
\begin{figure}
	\tikzstyle{vertex}=[circle, draw, inner sep=0pt, minimum size=3.5pt]
	\centering
	\subfloat[]{\begin{tikzpicture}[scale=0.9]
			\vertex[fill] (r1) at (0,0) [] {};
			\vertex[fill] (r2) at (0,.8) [] {};
			\vertex[fill] (r3) at (.8,1) [label=above:\scriptsize{$x_3$}] {};
			\vertex[fill] (r4) at (.8,-.2) [] {};
			\vertex[fill] (r5) at (1.3,.4) [] {};
			\vertex[fill] (r6) at (2,.8) [label=above:\scriptsize{$x_6$}]{};
			\vertex[fill] (r7) at (2,0) [] {};
			\vertex[fill] (r8) at (2.7,.4) [label=above:\scriptsize{$x_8$}] {};
			\vertex[fill] (r9) at (3.4,.8) [label=above:\scriptsize{$x_9$}] {};
			\vertex[fill] (r10) at (3.4,0) [] {};
			\vertex[fill] (r11) at (4.4,.8) [label=above:\scriptsize{$x_{11}$}] {};
			\vertex[fill] (r12) at (4.4,0) [] {};
			\vertex[fill] (r13) at (5.1,.4) [label=above:\scriptsize{$x_{13}$}] {};
			\vertex[fill] (r14) at (5.8,.8) [label=above:\scriptsize{$x_{14}$}] {};
			\vertex[fill] (r15) at (5.8,0) [] {};
			\vertex[fill] (r16) at (6.8,.8) [label=above:\scriptsize{$x_{16}$}] {};
			\vertex[fill] (r17) at (6.8,0) [] {};
			\vertex[fill] (r18) at (7.5,.4) [label=above:\scriptsize{$x_{18}$}] {};
			\vertex[fill] (r19) at (8.2,.8) [label=above:\scriptsize{$x_{19}$}] {};
			\vertex[fill] (r20) at (8.2,0) [] {};
			\vertex[fill] (r21) at (9.2,.8) [label=above:\scriptsize{$x_{21}$}] {};
			\vertex[fill] (r22) at (9.2,0) [] {};
			\vertex[fill] (r23) at (9.9,.4) [label=above:\scriptsize{$x_{23}$}] {};
			\vertex[fill] (r24) at (10.6,.8) [] {};
			\vertex[fill] (r25) at (10.6,0) [] {};
			\vertex[fill] (r26) at (11.6,.8) [] {};
			\vertex[fill] (r27) at (11.6,0) [] {};
			\vertex[fill] (r28) at (12.3,.4) [] {};
			\tikzstyle{vertex}=[circle, draw, inner sep=0pt, minimum size=0pt]
			\vertex (s) at (12.5,.6)[] {};
			\vertex (ss) at (12.5,.2) [] {};
			\vertex[] () at (1.3,.36) [label=above:\scriptsize{$x_5$}] {};
			\path
			(r5) edge (r2)
			(r1) edge (r2)
			(r1) edge (r5)
			(r1) edge (r3)
			(r1) edge (r4)
			(r2) edge (r3)
			(r2) edge (r4)
			(r3) edge (r6)
			(r4) edge (r3)
			(r4) edge (r7)
			(r5) edge (r7)
			(r5) edge (r6)			
			(r6) edge (r7)
			(r7) edge (r8)
			(r6) edge (r8)
			(r8) edge (r9)
			(r8) edge (r10)
			(r9) edge (r10)
			(r9) edge (r11)
			(r9) edge (r12)
			(r10) edge (r11)
			(r10) edge (r12)
			(r11) edge (r12)
			(r11) edge (r13)
			(r12) edge (r13)	
			(r13) edge (r14)
			(r13) edge (r15)
			(r14) edge (r15)
			(r14) edge (r16)
			(r14) edge (r17)
			(r15) edge (r16)
			(r15) edge (r17)
			(r16) edge (r17)
			(r18) edge (r17)
			(r16) edge (r18)
			(r18) edge (r19)
			(r18) edge (r20)
			(r19) edge (r20)
			(r19) edge (r21)
			(r19) edge (r22)
			(r20) edge (r21)
			(r20) edge (r22)
			(r21) edge (r23)
			(r21) edge (r22)
			(r22) edge (r23)
			(r23) edge (r24)
			(r23) edge (r25)
			(r24) edge (r25)
			(r24) edge (r26)
			(r24) edge (r27)
			(r25) edge (r26)
			(r25) edge (r27)
			(r26) edge (r27)
			(r26) edge (r28)
			(r27) edge (r28)
			(s) edge (r28)
			(ss) edge (r28);
		\end{tikzpicture}\label{fig:D31}}
	
	\subfloat[]{\begin{tikzpicture}[scale=0.9]
			\vertex[fill] (r1) at (0,0) [] {};
			\vertex[fill] (r2) at (0,.8) [label=above:\scriptsize{$x_1$}] {};
			\vertex[fill] (r3) at (.8,1) [label=above:\scriptsize{$x_3$}] {};
			\vertex[fill] (r4) at (.8,-.2) [] {};
			\vertex[fill] (r5) at (1.3,.4) [] {};
			\vertex[fill] (r6) at (2,.8) [label=above:\scriptsize{$x_6$}]{};
			\vertex[fill] (r7) at (2,0) [] {};
			\vertex[fill] (r8) at (2.7,.4) [label=above:\scriptsize{$x_8$}] {};
			\vertex[fill] (r9) at (3.4,.8) [label=above:\scriptsize{$x_9$}] {};
			\vertex[fill] (r10) at (3.4,0) [] {};
			\vertex[fill] (r11) at (4.4,.8) [label=above:\scriptsize{$x_{11}$}] {};
			\vertex[fill] (r12) at (4.4,0) [] {};
			\vertex[fill] (r13) at (5.1,.4) [label=above:\scriptsize{$x_{13}$}] {};
			\vertex[fill] (r14) at (5.8,.8) [label=above:\scriptsize{$x_{14}$}] {};
			\vertex[fill] (r15) at (5.8,0) [] {};
			\vertex[fill] (r16) at (6.8,.8) [label=above:\scriptsize{$x_{16}$}] {};
			\vertex[fill] (r17) at (6.8,0) [] {};
			\vertex[fill] (r18) at (7.5,.4) [label=above:\scriptsize{$x_{18}$}] {};
			\vertex[fill] (r19) at (8.2,.8) [label=above:\scriptsize{$x_{19}$}] {};
			\vertex[fill] (r20) at (8.2,0) [] {};
			\vertex[fill] (r21) at (9.2,.8) [label=above:\scriptsize{$x_{21}$}] {};
			\vertex[fill] (r22) at (9.2,0) [] {};
			\vertex[fill] (r23) at (9.9,.4) [label=above:\scriptsize{$x_{8}$}] {};
			\vertex[fill] (r24) at (10.6,.8) [label=above:\scriptsize{$x_9$}] {};
			\vertex[fill] (r25) at (10.6,0) [] {};
			\vertex[fill] (r26) at (11.6,.8) [label=above:\scriptsize{$x_{11}$}] {};
			\vertex[fill] (r27) at (11.6,0) [] {};
			\vertex[fill] (r28) at (12.3,.4) [label=above:\scriptsize{$x_{13}$}] {};
			\vertex[fill] (r29) at (13,.8) [label=above:\scriptsize{$x_{14}$}] {};
			\vertex[fill] (r30) at (13,0) [] {};
			\vertex[fill] (r31) at (14,.8) [label=above:\scriptsize{$x_{16}$}] {};
			\vertex[fill] (r32) at (14,0) [] {};
			\vertex[fill] (r33) at (14.7,.4) [label=above:\scriptsize{$x_{18}$}] {};
			\vertex[fill] (r34) at (15.4,.8) [label=above:\scriptsize{$x_{19}$}] {};
			\vertex[fill] (r35) at (15.4,0) [] {};
			\vertex[fill] (r36) at (16.4,.8) [label=above:\scriptsize{$x_{21}$}] {};
			\vertex[fill] (r37) at (16.4,0) [] {};
			\vertex[fill] (r38) at (17.1,.4) [label=above:\scriptsize{$x_{8}$}] {};
			\tikzstyle{vertex}=[circle, draw, inner sep=0pt, minimum size=0pt]
			\vertex (s) at (17.4,.6)[] {};
			\vertex (ss) at (17.4,.2) [] {};
			\vertex[] () at (1.3,.36) [label=above:\scriptsize{$x_5$}] {};
			\path
			(r5) edge (r2)
			(r1) edge (r2)
			(r1) edge (r5)
			(r1) edge (r3)
			(r1) edge (r4)
			(r2) edge (r3)
			(r2) edge (r4)
			(r3) edge (r6)
			(r4) edge (r3)
			(r4) edge (r7)
			(r5) edge (r7)
			(r5) edge (r6)			
			(r6) edge (r7)
			(r7) edge (r8)
			(r6) edge (r8)
			(r8) edge (r9)
			(r8) edge (r10)
			(r9) edge (r10)
			(r9) edge (r11)
			(r9) edge (r12)
			(r10) edge (r11)
			(r10) edge (r12)
			(r11) edge (r12)
			(r11) edge (r13)
			(r12) edge (r13)	
			(r13) edge (r14)
			(r13) edge (r15)
			(r14) edge (r15)
			(r14) edge (r16)
			(r14) edge (r17)
			(r15) edge (r16)
			(r15) edge (r17)
			(r16) edge (r17)
			(r18) edge (r17)
			(r16) edge (r18)
			(r18) edge (r19)
			(r18) edge (r20)
			(r19) edge (r20)
			(r19) edge (r21)
			(r19) edge (r22)
			(r20) edge (r21)
			(r20) edge (r22)
			(r21) edge (r23)
			(r21) edge (r22)
			(r22) edge (r23)
			(r23) edge (r24)
			(r23) edge (r25)
			(r24) edge (r25)
			(r24) edge (r26)
			(r24) edge (r27)
			(r25) edge (r26)
			(r25) edge (r27)
			(r26) edge (r27)
			(r26) edge (r28)
			(r27) edge (r28)
			(r28) edge (r29)
			(r28) edge (r30)
			(r29) edge (r30)
			(r29) edge (r31)
			(r29) edge (r32)
			(r30) edge (r31)
			(r30) edge (r32)
			(r31) edge (r32)
			(r31) edge (r33)
			(r32) edge (r33)
			(r34) edge (r33)
			(r35) edge (r33)
			(r34) edge (r35)
			(r34) edge (r36)
			(r34) edge (r37)
			(r35) edge (r36)
			(r35) edge (r37)
			(r36) edge (r37)
			(r36) edge (r38)
			(r37) edge (r38)
			(s) edge (r38)
			(ss) edge (r38)
			;
		\end{tikzpicture}\label{fig:D32}}
	\caption{The graph $G$ with at least one block $D_3$ and the components of an eigenvector of $\la=1$}
\end{figure}
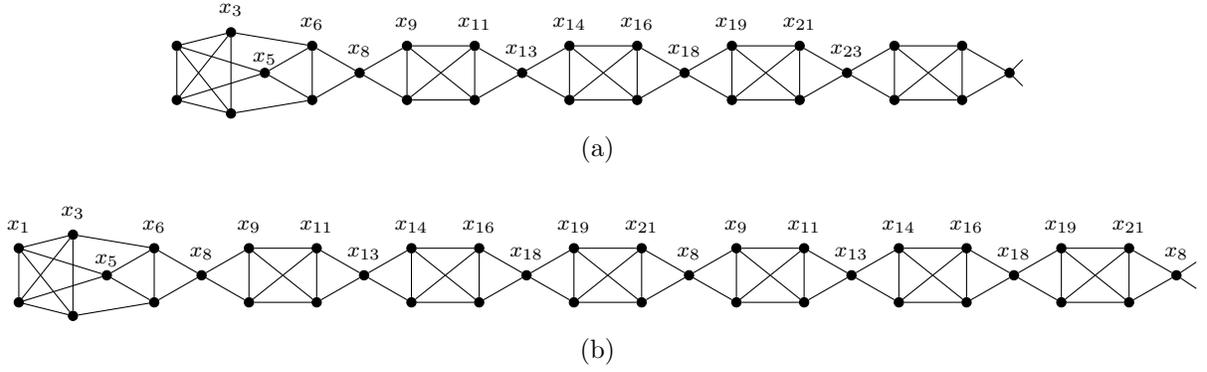

We first show that $\lambda\ne1$. For a contradiction, let $\la=1$. By the eigen-equation,  we can write the components of an eigenvector of $\la$ (as shown in Figure~\ref{fig:D31}) in terms of $x_1$.
It can be seen that  $x_6=x_{21}=-2x_1$ and $x_8=x_{23}=x_1$. Hence we must have $x_9=x_{24}$, $x_{11}=x_{26},$ and so on.
In this way, the components of $\x$ on the first three middle blocks of $G$ are repeated periodically, see Figure~\ref{fig:D32}.
In particular, the components of $\x$ on the cut vertices of $G$ are obtained in terms of $x_1$, as indicated in Figure~\ref{fig:D33}. Now suppose that the right end block of $G$ is $D_1$, see Figure~\ref{fig:D34}.
Since $x_{n-10}$ and $x_{n-5}$ are the components of two consecutive cut vertices, according to Figure~\ref{fig:D33}, one of the following occurs:
$$x_{n-5}=-5x_{n-10},\quad x_{n-5}=-\frac{4}{5}x_{n-10},\quad\hbox{or}\quad x_{n-5}=\frac{1}{4}x_{n-10}.$$
Again, by the eigen-equation, on the vertices of Figure~\ref{fig:D34}, we can write these components in terms of $x_n$.
It follows that $x_{n-10}=\frac{5}{2}x_n$ and $x_{n-5}=-\frac{7}{2}x_n$. So $x_{n-5}=-\frac{7}{5}x_{n-10}$, a contradiction. (Note that $x_1,x_n\neq 0$ since otherwise by the eigen-equation $\x=\bf{0}$.) A similar argument works if the right end block is one of $D_2, \ldots, D_5$.

In the following, we assume that both end blocks of $G$ are $D_3$. If $G$ has one $D_3$, it will be enough to use the same argument only on the submatrix corresponding to $D_3$.

\begin{figure}
	\tikzstyle{vertex}=[circle, draw, inner sep=0pt, minimum size=3.5pt]
	\centering
	\subfloat[]{\begin{tikzpicture}[scale=0.9]
			\vertex[fill] (r1) at (0,0) [] {};
			\vertex[fill] (r2) at (0,.8) [label=above:\scriptsize{$x_1$}] {};
			\vertex[fill] (r3) at (.8,1) [] {};
			\vertex[fill] (r4) at (.8,-.2) [] {};
			\vertex[fill] (r5) at (1.3,.4) [] {};
			\vertex[fill] (r6) at (2,.8) []{};
			\vertex[fill] (r7) at (2,0) [] {};
			\vertex[fill] (r8) at (2.7,.4) [label=above:\scriptsize{$x_1$}] {};
			\vertex[fill] (r9) at (3.4,.8) [] {};
			\vertex[fill] (r10) at (3.4,0) [] {};
			\vertex[fill] (r11) at (4.4,.8) [] {};
			\vertex[fill] (r12) at (4.4,0) [] {};
			\vertex[fill] (r13) at (5.1,.4) [label=above:\scriptsize{$-5x_1$}] {};
			\vertex[fill] (r14) at (5.8,.8) [] {};
			\vertex[fill] (r15) at (5.8,0) [] {};
			\vertex[fill] (r16) at (6.8,.8) [] {};
			\vertex[fill] (r17) at (6.8,0) [] {};
			\vertex[fill] (r18) at (7.5,.4) [label=above:\scriptsize{$4x_1$}] {};
			\vertex[fill] (r19) at (8.2,.8) [] {};
			\vertex[fill] (r20) at (8.2,0) [] {};
			\vertex[fill] (r21) at (9.2,.8) [] {};
			\vertex[fill] (r22) at (9.2,0) [] {};
			\vertex[fill] (r23) at (9.9,.4) [label=above:\scriptsize{$x_1$}] {};
			\vertex[fill] (r24) at (10.6,.8) [] {};
			\vertex[fill] (r25) at (10.6,0) [] {};
			\vertex[fill] (r26) at (11.6,.8) [] {};
			\vertex[fill] (r27) at (11.6,0) [] {};
			\vertex[fill] (r28) at (12.3,.4) [label=above:\scriptsize{$-5x_1$}] {};
			\vertex[fill] (r29) at (13,.8) [] {};
			\vertex[fill] (r30) at (13,0) [] {};
			\vertex[fill] (r31) at (14,.8) [] {};
			\vertex[fill] (r32) at (14,0) [] {};
			\vertex[fill] (r33) at (14.7,.4) [label=above:\scriptsize{$4x_1$}] {};
			\tikzstyle{vertex}=[circle, draw, inner sep=0pt, minimum size=0pt]
			\vertex (s) at (15,.6)[] {};
			\vertex (ss) at (15,.2) [] {};
			\path
			(r5) edge (r2)
			(r1) edge (r2)
			(r1) edge (r5)
			(r1) edge (r3)
			(r1) edge (r4)
			(r2) edge (r3)
			(r2) edge (r4)
			(r3) edge (r6)
			(r4) edge (r3)
			(r4) edge (r7)
			(r5) edge (r7)
			(r5) edge (r6)			
			(r6) edge (r7)
			(r7) edge (r8)
			(r6) edge (r8)
			(r8) edge (r9)
			(r8) edge (r10)
			(r9) edge (r10)
			(r9) edge (r11)
			(r9) edge (r12)
			(r10) edge (r11)
			(r10) edge (r12)
			(r11) edge (r12)
			(r11) edge (r13)
			(r12) edge (r13)	
			(r13) edge (r14)
			(r13) edge (r15)
			(r14) edge (r15)
			(r14) edge (r16)
			(r14) edge (r17)
			(r15) edge (r16)
			(r15) edge (r17)
			(r16) edge (r17)
			(r18) edge (r17)
			(r16) edge (r18)
			(r18) edge (r19)
			(r18) edge (r20)
			(r19) edge (r20)
			(r19) edge (r21)
			(r19) edge (r22)
			(r20) edge (r21)
			(r20) edge (r22)
			(r21) edge (r23)
			(r21) edge (r22)
			(r22) edge (r23)
			(r23) edge (r24)
			(r23) edge (r25)
			(r24) edge (r25)
			(r24) edge (r26)
			(r24) edge (r27)
			(r25) edge (r26)
			(r25) edge (r27)
			(r26) edge (r27)
			(r26) edge (r28)
			(r27) edge (r28)
			(r28) edge (r29)
			(r28) edge (r30)
			(r29) edge (r30)
			(r29) edge (r31)
			(r29) edge (r32)
			(r30) edge (r31)
			(r30) edge (r32)
			(r31) edge (r32)
			(r31) edge (r33)
			(r32) edge (r33)
			(s) edge (r33)
			(ss) edge (r33)	 ;
		\end{tikzpicture}\label{fig:D33}}
	
	\subfloat[]{\begin{tikzpicture}[scale=.95]
			\vertex[fill] (22) at (5.9,.4) [label=above:\scriptsize{$x_{n-10}\ \ \ \ $}] {};
			\vertex[fill] (23) at (6.4,.8) [label=above:\scriptsize{$x_{n-9}$}] {};
			\vertex[fill] (24) at (6.4,0) [] {};
			\vertex[fill] (25) at (7.4,.8) [label=above:\scriptsize{$x_{n-7}$}] {};
			\vertex[fill] (26) at (7.4,0) [] {};
			\vertex[fill] (27) at (8,.4) [] {};
			\vertex[fill] (28) at (8.6,.8) [label=above:\scriptsize{$x_{n-4}$}] {};
			\vertex[fill] (29) at (8.6,0) [] {};
			\vertex[fill] (30) at (9.6,.8) [label=above:\scriptsize{$x_{n-2}$}] {};
			\vertex[fill] (31) at (9.6,0) [] {};
			\vertex[fill] (32) at (10.1,.4) [] {};
			\tikzstyle{vertex}=[circle, draw, inner sep=0pt, minimum size=0pt]
			\vertex[] (s) at (5.6,.2)[] {};
			\vertex[] (ss) at (5.6,.6)[] {};
			\vertex[] () at (8,.6) [label=above:\scriptsize{$x_{n-5}$}] {};
			\path
			(22) edge (23)
			(22) edge (24)
			(23) edge (24)
			(23) edge (25)
			(23) edge (26)
			(24) edge (25)
			(24) edge (26)
			(25) edge (26)
			(25) edge (27)
			(26) edge (27)
			(27) edge (28)
			(27) edge (29)
			(28) edge (30)
			(28) edge (31)
			(28) edge (32)
			(29) edge (30)
			(29) edge (31)
			(29) edge (32)
			(30) edge (31)
			(30) edge (32)
			(31) edge (32)
			(22) edge (s)
			(22) edge (ss) ;
		\end{tikzpicture}\label{fig:D34}}
	\caption{The graph $G$ with at least one block $D_3$ and the components of an eigenvector of $\la=1$ on cut vertices}
\end{figure}
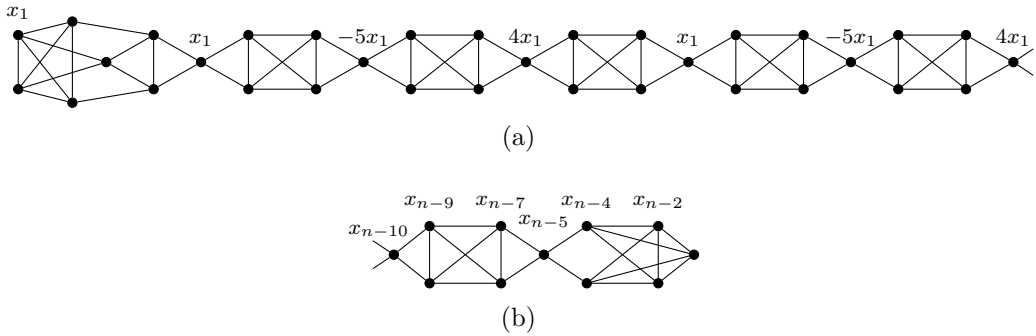 

We have
$$Q-\lambda I=\begin{tikzpicture}[baseline=(current bounding box.center)]
	\matrix (m) [matrix of math nodes,nodes in empty cells,right delimiter={]},left delimiter={[} ]{
		1-\lambda & 2 & 1 & 0 & & & & & & \\
		2 & 1-\lambda & 0 & 1 & & & & & &\\
		2 & 0 & -\lambda & 2  & & & & & & \\
		0 & 1 &1  & 1-\lambda  & & & & & &\\
		& & &  &  & & & & & \\
		& & &  &  & & & & & \\
		& & & &  &  &1-\lambda & 1 &1 & 0   \\
		& & & & & & 2 & -\lambda & 0 & 2  \\
		& & & & & &1 &0 &1-\lambda & 2  \\
		& & & & & &  0&1 &2 &1-\lambda   \\
	} ;
	\draw[line width=0.3mm, loosely  dotted] (m-4-4)-- (m-7-7);
	\draw[line width=0.3mm,loosely dotted] (m-4-3)-- (m-8-7);
	\draw[line width=0.3mm,loosely dotted] (m-3-4)-- (m-7-8);
\end{tikzpicture}_{m\times m}.$$

Now we perform the following elementary operations on the rows $R_i$ and the columns $C_i$ as follows: add $\frac{-2}{1-\lambda}R_1$ to $R_3$,  add $\frac{1-\lambda}{4}R_3$ to $R_4$,  add
$\frac{-1}{2}R_{m-2}$ to $R_{m-1}$, and finally add $\frac{-1}{1-\lambda}C_m$ to $C_{m-2}$.
The resulting matrix is
$$Q'=\begin{tikzpicture}[baseline=(current bounding box.center)]
	\matrix (m) [matrix of math nodes,nodes in empty cells,right delimiter={]},left delimiter={[} ]{
		1-\lambda & 2 & 1 & 0 & & & & & & \\
		2 & 1-\lambda & 0 & 1 & & & & & &\\
		0 & \frac{-4}{1-\lambda} & \frac{\lambda^2-\lambda-2}{1-\lambda   } & 2  & & & & & & \\
		0 & 0 & \frac{\lambda^2-\lambda+2}{4} & {\frac{3-3\lambda}{2}}   & & & & & &\\
		& & &  &  & & & & & \\
		& & &  &  & & & & & \\
		& & & &  &  & 1-\lambda & 1 &1 &0   \\
		& & & & & &  2 & \frac{\lambda^2-\lambda-2}{1-\lambda   } & 0 & 2   \\
		& & & & & &0 &\frac{-\lambda^2+\lambda-2}{2-2\lambda   } &1-\lambda & 2  \\
		& & & & & &  0&0 &2 &1-\lambda  \\
	} ;
	\draw[line width=0.3mm, loosely  dotted] (m-4-4)-- (m-7-7);
	\draw[line width=0.3mm,loosely dotted] (m-4-3)-- (m-8-7);
	\draw[line width=0.3mm,loosely dotted] (m-3-4)-- (m-7-8);
\end{tikzpicture}.$$
So ${\rm rank}(Q-\lambda I)={\rm rank}(Q')$.
Since  $\lambda^2-\lambda+2\ne0$ for real $\la$, the submatrix obtained by deleting the first row and the last column of $Q'$ is upper triangular with non-zero diagonal.
This implies that rank$(Q')\ge m-1$ and
so the nullity of $Q-\la I$ is at most  $1$. Therefore, $\lambda$ is a simple eigenvalue of $Q$.	
\end{proof}

In the next theorem, we show that the eigenvalues of graphs of $\Fq^*$ with eigenvectors of second type come from a small set of size eight.
This, in view of Lemma~\ref{lem:1stTypeQua}, implies that non-simple eigenvalues are restricted to at most eight values.

\begin{theorem}\label{thm:simple}
Let $G\in\Fq^*$  be a graph of order $n\geq 11$. Then non-simple  eigenvalues of $G$  belong to 	
$$\Lambda:=\left\{-2,0, \pm 1, -1\pm\sqrt{2}, (-1\pm\sqrt{5})/2\right\}.$$
\end{theorem}
\begin{proof}
Let $\lambda$ be an eigenvalue of $G$ with eigenvector $\x$ of the second type. In view of Lemma~\ref{lem:2ndtypeFc}, it suffices to show that $\la$ belongs to $\Lambda$.

 The components of $\x$ sum up to zero on each cell of  $\Pi$. In particular, the components of $\x$ on the cut vertices are zero.
Note that $\x\neq 0$ and we can always find a block $B$ such that  the components of $\x$ are not all zero on it.
Firstly, let $B$ be  a middle block. The components of $\x$ on $B$ are as shown in Figure~\ref{fig:lem:type2}.
Since the vertices $i$ and ${i+1}$ belong to the same cell in $\Pi$ and $\x$ is of the second type, we have $x_i+x_{i+1}=0$.
By using the eigen-equation, it can be seen that $\lambda  x_i=-x_i$. If $x_i\neq 0$, then  $\lambda=-1$, otherwise
$x_{i+2}\neq 0$ and from $\lambda  x_{i+2}=-x_{i+2}$ we again obtain $\lambda=-1$.
\begin{figure}[h!]
\centering
\begin{tikzpicture}[scale=0.9]
	\vertex[fill] (r) at (0,0) [label=above:\scriptsize{$0$}] {};
	\vertex[fill] (r1) at (.5,.5) [label=above:\scriptsize{$x_{i}$}] {};
	\vertex[fill] (r2) at (.5,-.5) [label=below:\scriptsize{$-x_{i}=x_{i+1}\ \ \ \ \ \ \ \ \ \ $}] {};
	\vertex[fill] (r3) at (1.5,.5) [label=above:\scriptsize{$x_{i+2}$}] {};
	\vertex[fill] (r4) at (1.5,-.5) [label=below:\scriptsize{$\ \ \ \ \ \ \ \ \ \ \ x_{i+3}=-x_{i+2}$}] {};
    \vertex[fill] (r5) at (2,0) [label=above:\scriptsize{$0$}] {};
	\path
		(r) edge (r1)
		(r) edge (r2)
		(r1) edge (r2)
		(r1) edge (r3)
        (r1) edge (r4)
		(r2) edge (r3)
	    (r2) edge (r4)
	    (r3) edge (r4)
	    (r5) edge (r4)
	    (r3) edge (r5);
\end{tikzpicture}
\caption{A middle block of $G$ and the components of an eigenvector of the second type}\label{fig:lem:type2}
\end{figure}
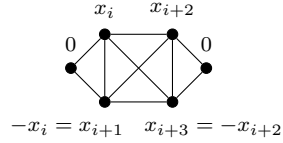
Now we examine the end blocks of $G$.
The end blocks and the components of $\x$ on them are illustrated
  in Figure~\ref{fig:proof}. Accordingly, one of the following cases (i)--(v) occurs.
\begin{itemize}
\item[(i)]  $B=D_1$.  Since $\x$ is of the second type, we have $x_1+x_2+x_3=0$.  If $x_4\neq 0$, then $\lambda x_4=x_1+x_2+x_3=0$, so $\lambda=0$.  If $x_4=0$, then at least one of $x_1,x_2,x_3$ must be non-zero. By the symmetry, we may suppose that $x_1\neq 0$. So  from  $\lambda x_1=x_2+x_3=-x_1$, we have that  $\lambda=-1$.
\item[(ii)] $B=D_2$. We have $x_1+x_2+x_3+x_4=0$. First, suppose $x_5\neq 0$. Then from $\lambda x_1=-x_1+x_5$ and $\lambda x_3=-x_3+x_5$, we obtain either $\lambda=-1$ or $0\neq x_1=x_3$.  If the latter holds, then $\lambda x_5=-x_5+2x_1$. From
$\lambda x_1=x_5-x_1$, then it follows that $\lambda^2+2\lambda -1=0$, that is
$\lambda=-1\pm\sqrt{2}$.
Next, suppose $x_5=0$. Then at least one of $x_1,\ldots,x_4$ must be non-zero. By the symmetry, we may assume that $x_1\neq 0$. Hence, from  $\lambda x_1=x_2+x_3+x_4=-x_1$, we have $\lambda=-1$.
\item[(iii)] $B=D_3$.
If $x_5\neq 0$, then $\lambda x_5=0$, so $\lambda=0$. Suppose that $x_5=0$. If
$x_1\neq 0$, then  from $\lambda x_1=-x_1$, we have $\lambda=-1$. If $x_1=0$, then from
$\lambda x_3=x_6-x_3$ and $\lambda x_6=x_3-x_6$,  it follows that either $\lambda=0$ or $0\neq x_3=-x_6$ and subsequently $\lambda=-2$.
\item[(iv)]  $B=D_4$.
We have $x_1+x_2+x_3+x_4=0$.  If $x_7\neq 0$, then $\lambda x_7=-x_7$, so $\lambda=-1$. Now, suppose that $x_7=0$. If $x_1=0$, then
$\lambda x_1=-x_1+x_5$ implies that  $x_5=0$ and then  $x_3=\lambda x_5=0$.
Since  all the components of $\x$ on the first cell cannot be zero, we have $0\neq x_2=-x_4$. Therefore,
$\lambda x_2=-x_2$ implies that  $\lambda=-1$. Next, suppose that  $x_1\neq 0$. If $x_5=0$,  then from
$\lambda x_1=-x_1+x_5$,  we have  $\lambda=-1$.  If $x_5\neq 0$, then  from $\lambda x_1=-x_1+x_5$ and
$\lambda x_3=-x_3+x_5$, we conclude that  $\lambda=-1$ or $0\neq x_1=x_3$. If the latter holds, then from $\lambda x_1=-x_1+x_5$ and $\lambda x_5=2x_1$, we obtain $\lambda=1$ or $\lambda=-2$.
\item[(v)] $B=D_5$.
We have  $x_1+x_2+x_3=0$. If $x_8\neq 0$, then $\lambda x_8=-x_8$ implying that $\lambda=-1$. Thus, assume that $x_8=0$. If all the components of $\x$ on the first cell are  zero, then from  $\lambda x_4=x_6$ and $\lambda x_6=-x_6+x_4$, we see that $x_4, x_6\neq 0$.  Now from $\lambda x_4=x_6$ and $\lambda x_6=-x_6+x_4$, it follows that
$\lambda^2+\lambda-1=0$, so
$\lambda=(-1\pm\sqrt{5})/{2}$. Otherwise, by the symmetry, we may suppose that $x_1\neq 0$.  Hence, from $\lambda x_1=-x_1$, we obtain $\lambda=-1$.
\end{itemize}
The proof is now complete.
  \end{proof}
\begin{figure}[h!]
	\captionsetup[subfigure]{labelformat=empty}
	\centering
	\subfloat[$D_1$]{\begin{tikzpicture}[scale=0.9]
			\vertex[fill] (r) at (0,0) [label=above:\scriptsize{$x_1$}] {};
			\vertex[fill] (r1) at (.5,.5) [label=above:\scriptsize{$x_2$}] {};
			\vertex[fill] (r2) at (.5,-.5) [label=below:\scriptsize{$x_3$}] {};
			\vertex[fill] (r3) at (1.5,.5) [label=above:\scriptsize{$x_4$}] {};
			\vertex[fill] (r4) at (1.5,-.5) [label=below:\scriptsize{$-x_4$}] {};
			\vertex[fill] (r5) at (2,0) [label=above:\scriptsize{$0$}] {};
			\path
			(r) edge (r1)
			(r) edge (r2)
			(r) edge (r3)
			(r) edge (r4)
			(r1) edge (r2)
			(r1) edge (r3)
			(r1) edge (r4)
			(r2) edge (r3)
			(r2) edge (r4)
			(r5) edge (r4)
			(r5) edge (r3)
			;
	\end{tikzpicture}}
	~~
	\subfloat[$D_2$]{\begin{tikzpicture}[scale=0.9]
			\vertex[fill] (r1) at (0,.8) [label=above:\scriptsize{$x_1$}] {};
			\vertex[fill] (r2) at (.8,1) [label=above:\scriptsize{$x_3$}] {};
			\vertex[fill] (r3) at (0,0) [label=below:\scriptsize{$x_2$}] {};
			\vertex[fill] (r4) at (.8,-.2) [label=below:\scriptsize{$x_4$}] {};
			\vertex[fill] (r5) at (1.6,.8) [label=above:\scriptsize{$x_5$}] {};
			\vertex[fill] (r6) at (1.6,0) [label=below:\scriptsize{$-x_5$}] {};
			\vertex[fill] (r7) at (2.3,.4) [label=above:\scriptsize{$0$}] {};
			\path
			(r5) edge (r2)
			(r1) edge (r2)
			(r1) edge (r5)
			(r1) edge (r3)
			(r1) edge (r4)
			(r2) edge (r3)
			(r2) edge (r4)
			(r3) edge (r6)
			(r4) edge (r3)
			(r4) edge (r6)
			(r5) edge (r6)
			(r7) edge (r6)
			(r5) edge (r7)
			;
	\end{tikzpicture}}
	~~
	\subfloat[$D_3$]{\begin{tikzpicture}[scale=0.8]
			\vertex[fill] (r1) at (0,0) [label=below:\scriptsize{$-x_1$}] {};
			\vertex[fill] (r2) at (0,.8) [label=above:\scriptsize{$x_1$}] {};
			\vertex[fill] (r3) at (.8,1) [label=above:\scriptsize{$x_3$}] {};
			\vertex[fill] (r4) at (.8,-.2) [label=below:\scriptsize{$-x_3$}] {};
			\vertex[fill] (r5) at (1.3,.4) [label=above:\scriptsize{$x_5$}] {};
			\vertex[fill] (r6) at (2,.8) [label=above:\scriptsize{$x_6$}] {};
			\vertex[fill] (r7) at (2,0) [label=below:\scriptsize{$-x_6$}] {};
			\vertex[fill] (r8) at (2.7,.4) [label=above:\scriptsize{$0$}] {};
			\path
			(r5) edge (r2)
			(r1) edge (r2)
			(r1) edge (r5)
			(r1) edge (r3)
			(r1) edge (r4)
			(r2) edge (r3)
			(r2) edge (r4)
			(r3) edge (r6)
			(r4) edge (r3)
			(r4) edge (r7)
			(r5) edge (r7)
			(r5) edge (r6)
			(r6) edge (r7)
			(r7) edge (r8)
			(r6) edge (r8)  ;
	\end{tikzpicture}}
	~~
	\subfloat[$D_4$]{\begin{tikzpicture}[scale=0.8]
			\vertex[fill] (r1) at (0,1) [label=above:\scriptsize{$x_1$}] {};
			\vertex[fill] (r2) at (.8,1.2) [label=above:\scriptsize{$x_3$}] {};
			\vertex[fill] (r3) at (0,0) [label=below:\scriptsize{$x_2$}] {};
			\vertex[fill] (r4) at (.8,-.2) [label=below:\scriptsize{$x_4$}] {};
			\vertex[fill] (r5) at (1.6,1) [label=above:\scriptsize{$x_5$}] {};
			\vertex[fill] (r6) at (1.6,0) [label=below:\scriptsize{$-x_5$}] {};
			\vertex[fill] (r7) at (2.6,1) [label=above:\scriptsize{$x_7$}] {};
			\vertex[fill] (r8) at (2.6,0) [label=below:\scriptsize{$-x_7$}] {};
			\vertex[fill] (r9) at (3.1,.5) [label=above:\scriptsize{$0$}] {};
			\path
			(r5) edge (r2)
			(r1) edge (r2)
			(r1) edge (r5)
			(r1) edge (r3)
			(r1) edge (r4)
			(r2) edge (r3)
			(r2) edge (r4)
			(r3) edge (r6)
			(r4) edge (r3)
			(r4) edge (r6)
			(r5) edge (r8)
			(r5) edge (r7)
			(r6) edge (r8)
			(r6) edge (r7)
			(r7) edge (r8)
			(r9) edge (r7)
			(r9) edge (r8)
			;
	\end{tikzpicture}}
	~~
	\subfloat[$D_5$]{\begin{tikzpicture}[scale=0.8]
			\vertex[fill] (r1) at (0,0) [label=above:\scriptsize{$x_1$}] {};
			\vertex[fill] (r3) at (.5,.5) [label=above:\scriptsize{$x_2$}] {};
			\vertex[fill] (r2) at (.5,-.5) [label=below:\scriptsize{$x_3$}] {};
			\vertex[fill] (r5) at (1.5,.5) [label=above:\scriptsize{$x_4$}] {};
			\vertex[fill] (r4) at (1.5,-.5) [label=below:\scriptsize{$-x_4$}] {};
			\vertex[fill] (r6) at (2.3,-.5) [label=below:\scriptsize{$-x_6$}] {};
			\vertex[fill] (r7) at (2.3,.5) [label=above:\scriptsize{$x_6$}] {};
			\vertex[fill] (r8) at (3.3,-.5) [label=below:\scriptsize{$-x_8$}]{};
			\vertex[fill] (r9) at (3.3,.5) [label=above:\scriptsize{$x_8$}] {};
			\vertex[fill] (r10) at (3.8,0) [label=above:\scriptsize{$0$}] {};
			\path
			(r1) edge (r2)
			(r1) edge (r3)
			(r1) edge (r4)
			(r1) edge (r5)
			(r2) edge (r3)
			(r2) edge (r4)
			(r2) edge (r5)
			(r3) edge (r4)
			(r3) edge (r5)
			(r6) edge (r4)
			(r5) edge (r7)
			(r6) edge (r7)
			(r6) edge (r8)
			(r6) edge (r9)
			(r8) edge (r7)
			(r7) edge (r9)
			(r8) edge (r9)
			(r8) edge (r10)
			(r9) edge (r10)
			;
	\end{tikzpicture}}
	\caption{End blocks of $G$ and the components of an eigenvector of the second type} \label{fig:proof}
\end{figure}
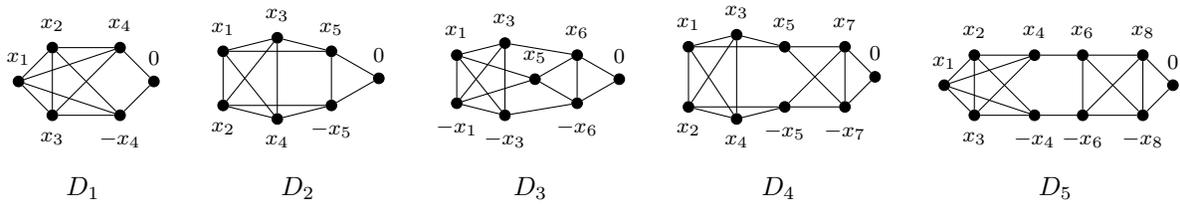

Conjecture~\ref{conj:MinQuartic}, if true, yields an improvement on Theorem~\ref{thm:simple} as follows.
From the proof of Theorem~\ref{thm:simple}, we observe that, the multiplicity of  the eigenvalue $-1$ is at least twice the number of the middle blocks of a  quartic graph $G\in\Fq^*$.
  So it can be seen that, the multiplicity of the eigenvalue $-1$ is at least
 $(2n-38)/5$.
 If Conjecture~\ref{conj:MinQuartic} is true, then the minimal quartic graph $G$ would  not contain $D_4$. Also $D_3$ and $D_5$ can only appear at most once as an end block of $G$. From the  proof of Theorem~\ref{thm:simple},   and  by the fact that
 $1$ and $(-1\pm\sqrt{5})/2$ cannot be an eigenvalue of $Q$ 
  (by a similar argument as in the proof of Theorem~\ref{thm:simplecubic}),
 it follows that the multiplicity of eigenvalues $1$ and $(-1\pm\sqrt{5})/2$ are at most one. Thus the non-simple eigenvalues of $G$ will be restricted to  $\left\{0, -1, -1\pm\sqrt2\right\}$.

Similarly, for cubic graphs with minimum spectral gap, the multiplicity of  the eigenvalue $-1$ is at least equal to the number of the middle blocks of $G$. So it is at least $(n-12)/4$. Therefore  $-1$ is always an eigenvalue of  cubic and quartic  graphs with minimum spectral gap which has the maximum multiplicity.

 \section{Gap intervals for minimal regular graphs}\label{sec:gap}

In this section, we determine the gap intervals for the  graphs $\de_n$ and $\ga_n$.
We also show that the smallest eigenvalues of $\de_n$ and $\ga_n$ are convergent as $n$ grows and specify the limits.

\subsection{Gap intervals for  the spectrum of $\de_n$}

 \begin{remark} \label{remark2}\rm
Let  $\lambda \notin \lbrace -1, 0\rbrace$ be an eigenvalue of  $\de_n$ with eigenvector $\x=(x_1, \ldots, x_n)$.
By the symmetry of $\de_n$, $\x'=(x_n, \ldots, x_1)$  is also an eigenvector for $\lambda$.
By Theorem~\ref{thm:simplecubic},  $\lambda$ is simple. It follows that  $\x=\pm\x'$, implying that
\begin{equation}\label{eq:pm}
 |x_i|=|x_{n-i+1}|, ~{\text{for}}~i=1, \ldots, n.
\end{equation}
 \end{remark}
We start with a lower bound on the smallest eigenvalue of $\de_n$.
 \begin{theorem} \label{thm:Gapcubic1}
If $n\geq 14$, then   $\rho(\de_n)\geq -1-\sqrt{2}$.
\end{theorem}
\begin{proof}
For convenience, let us set $G:=\de_n$ and $\rho:=\rho(G)$.
If $n=14$ or $18$, then by direct computation we see that  $\rho\geq -1-\sqrt{2}$. So we suppose that $n\geq 22$.
Since $G$ is not a union of complete graphs, we have $\rho <-1$.
 So, by the proof of Theorem~\ref{thm:simplecubic}, $\rho $ has a unit eigenvector $\x$ that is constant on each cell of the equitable partition $\Pi$ of $G$ and further, by Remark~\ref{remark2}, the components of $\x$ satisfy \eqref{eq:pm}.

 The first step of the proof involves decomposing $A(G)$ as the sum of a positive definite matrix
	and two other `simpler' matrices.

Consider a set of cliques ($K_2$  and $K_3$) of $G$ that covers the edges of $G$, as shown in Figure~\ref{fig:clique2}.
 \begin{figure}[h!]
 \centering
\begin{tikzpicture}[scale=1.1]
	\vertex[fill] (1) at (0,-.5) [label=left:\scriptsize{$x_2$}] {};
	\vertex[fill] (2) at (0,.5) [label=left:\scriptsize{$x_1$}] {};
	\vertex[fill] (3) at (1,-.5) [label=below:\scriptsize{$x_4$}] {};
	\vertex[fill] (4) at (1,.5) [label=above:\scriptsize{$x_3$}] {};
    \vertex[fill] (5) at (1.5,0) [label=above:\scriptsize{$x_5$}] {};
    \vertex[fill] (6) at (2,0) [] {};
    \vertex[fill] (7) at (2.5,.5) [] {};
    \vertex[fill] (8) at (2.5,-.5) [] {};
    \vertex[fill] (9) at (3,0) [] {};
    \vertex[fill] (26) at (4.9,0) [] {};
    \vertex[fill] (27) at (5.4,.5) [] {};
    \vertex[fill] (28) at (5.4,-.5) [] {};
    \vertex[fill] (29) at (5.9,0) [] {};
    \vertex[fill] (34) at (7.9,-.5) [label=right:\scriptsize{$x_n$}] {};
	\vertex[fill] (33) at (7.9,.5) [label=right:\scriptsize{$x_{n-1}$}] {};
	\vertex[fill] (32) at (6.9,-.5) [label=below:\scriptsize{$x_{n-2}$}] {};
	\vertex[fill] (31) at (6.9,.5) [label=above:\scriptsize{$x_{n-3}$}] {};
    \vertex[fill] (30) at (6.4,0) [label=above:\scriptsize{$x_{n-4}\ \ \ $}] {};
    \tikzstyle{vertex}=[circle, draw, inner sep=0pt, minimum size=1pt]
     \vertex[fill] (16) at (3.75,0) [] {};
      \vertex[fill] (18) at (3.95,0) [] {};
    \vertex[fill] (20) at (4.15,0) [] {};
      \tikzstyle{vertex}=[circle, draw, inner sep=0pt, minimum size=0pt]
       \vertex[fill] (14) at (3.3,0) [] {};
    \vertex[fill] (21) at (4.6,0) [] {};
	\path[densely dashed,very thick]
		(1) edge (2);
	\path[densely dashed]
		(1) edge (3)
	    (1) edge (4)
		(2) edge (3)
        (2) edge (4);
        \path
    	(3) edge (5)
	    (4) edge (5)
	   (5) edge (6);
	    \path[densely dashed]
	    (6) edge (7)
	    (6) edge (8);
	    \path[densely dashed,very thick]
	    (7) edge (8);
	    \path[densely dashed]
	    (7) edge (9)
	    (8) edge (9);
	    \path
     (14) edge (9)
	   (26) edge (21);
	   \path[densely dashed]
	    (26) edge (27)
	    (26) edge (28);
	    \path[densely dashed,very thick]
	    (27) edge (28);
	    \path[densely dashed]
	    (27) edge (29)
	    (28) edge (29);
	    \path
	    (29) edge (30)
	    (30) edge (31)
        (30) edge (32);
        \path[densely dashed]
        (31) edge (33)
        (31) edge (34)
	    (32) edge (33)
        (32) edge (34);
       \path[densely dashed,very thick]
        (33) edge (34);
	    \end{tikzpicture}
\caption{The edge clique cover of  $G$: each normal line indicates a $K_2$, each dashed triangle is a $K_3$; thick dashed edges are covered by two $K_3$.} \label{fig:clique2}
\end{figure}
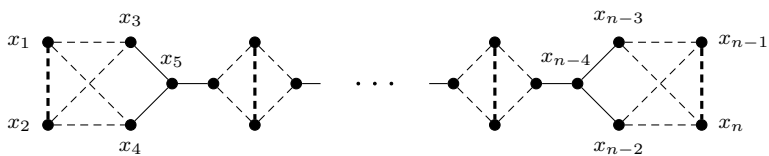

 Let $H$ be the spanning subgraph of $G$ induced by
the thick dashed edges,
 see  Figure~\ref{fig:Hcubic}.
  \begin{figure}[h!]
 \centering
\begin{tikzpicture}[scale=1]
	\vertex[fill] (1) at (0,-.5) [] {};
	\vertex[fill] (2) at (0,.5) [] {};
	\vertex[fill] (3) at (1,-.5) [] {};
	\vertex[fill] (4) at (1,.5) [] {};
    \vertex[fill] (5) at (1.5,0) [] {};
    \vertex[fill] (6) at (2,0) [] {};
    \vertex[fill] (7) at (2.5,.5) [] {};
    \vertex[fill] (8) at (2.5,-.5) [] {};
    \vertex[fill] (9) at (3,0) [] {};
    \vertex[fill] (26) at (4.9,0) [] {};
    \vertex[fill] (27) at (5.4,.5) [] {};
    \vertex[fill] (28) at (5.4,-.5) [] {};
    \vertex[fill] (29) at (5.9,0) [] {};
    \vertex[fill] (34) at (7.9,-.5) [] {};
	\vertex[fill] (33) at (7.9,.5) [] {};
	\vertex[fill] (32) at (6.9,-.5) [] {};
	\vertex[fill] (31) at (6.9,.5) [] {};
    \vertex[fill] (30) at (6.4,0) [] {};
      \tikzstyle{vertex}=[circle, draw, inner sep=0pt, minimum size=1pt]
    \vertex[fill] (16) at (3.75,0) [] {};
    \vertex[fill] (18) at (3.95,0) [] {};
    \vertex[fill] (20) at (4.15,0) [] {};
	\path
		(1) edge (2)
	    (7) edge (8)
	    (27) edge (28)
        (33) edge (34) ;
\end{tikzpicture}
\caption{The  subgraph $H$ of $G$} \label{fig:Hcubic}
\end{figure}
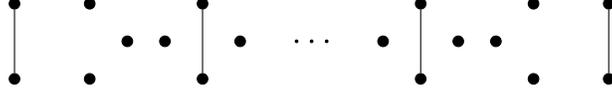

  Let $M$ be the vertex-clique incidence matrix of $G$. So, the rows and columns of $M$ are indexed by the vertices and the non-trivial cliques of $G$, respectively, and its $(i,j)$-entry is $1$ if $i$-th vertex belongs to the $j$-th  clique and $0$ otherwise.
  Also let $A$ and $B$ be  the adjacency matrices of $G$ and $H$, respectively, and $C$ be the matrix whose entries are all zero except for $C_{5,5}=C_{n-4,n-4}=1$.
  Note that the $(i,j)$-entry of $MM^\top$ is equal to the inner product of the two rows of $M$ corresponding to the vertices $i$ and $j$. Each vertex of $G$ belongs to two cliques of the above clique covering except for the vertices $5$ and $n-4$ which appear in three cliques. This implies that the diagonal entries of $MM^\top$ and $2I+C$ are the same. Now, if $ij \notin  E(G)$,
   $i$ and $j$ do not appear in any clique together. If $ij \in E(G)\backslash E(H)$, then $i$ and $j$ appear in exactly one clique together and if $ij \in E(H)$, they appear in two cliques together. From this argument, it follows that the off-diagonal entries of $MM^\top$ coincide with those of $A+B$.  It turns out that
$$MM^\top=2I+A+B+C.$$
Therefore,
$$\rho=\x A\x^\top=(\x M)(\x M)^\top-2-\x B\x^\top-\x C\x^\top.$$
We have
$\x B\x^\top=2\sum_{ij\in E(H)}x_ix_j$
and since $|x_5|=|x_{n-4}|$,  $\x C\x^\top=2x_5^2.$
It follows that
\begin{equation}\label{eq:lambda1cubic}
\rho \geq -2-2\sum_{ij\in E(H)}x_ix_j-2x_5^2.
\end{equation}

 The second step of the proof involves
finding an estimate for $\sum_{ij\in E(H)}x_ix_j$ in terms of
first few $x_i$'s which in turn results in an inequality expressed as a rational function in $\rho$ from which  in the final step we drive the desired lower bound on $\rho$.

Note that if $ij\in E(H)$, then $x_ix_j=x_i^2$.
Suppose that $x_r, x_s,$  and $x_t$ are the components of $\x$ on the cells of a middle block of $G$, as depicted in
Figure~\ref{fig:1middle}.
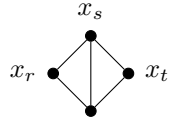
\begin{figure}[h!]
 \centering
\begin{tikzpicture}[scale=1]
    \vertex[fill] (6) at (2,0) [label=left:\footnotesize{$x_r$}] {};
    \vertex[fill] (7) at (2.5,.5) [label=above:\footnotesize{$x_s$}] {};
    \vertex[fill] (8) at (2.5,-.5) [] {};
    \vertex[fill] (9) at (3,0) [label=right:\footnotesize{$x_t$}] {};
	\path
	    (6) edge (7)
	    (6) edge (8)
	    (7) edge (8)
	    (7) edge (9)
	    (8) edge (9)  ;
\end{tikzpicture}
\caption{A middle block of  $G$ and the components of $\x$} \label{fig:1middle}
\end{figure}

By using the eigen-equation on the vertex $s$, we have
$x_s=(x_r+x_t)/(\rho-1)$.
  Then we see that $x_s^2\leq k(x_r^2+x_t^2)$, in which $k=2/(\rho-1)^2$. Let $S$ be the set of cut vertices of $G$, so
  \begin{align*}
  \sum_{ij\in E(H)}x_ix_j &\leq k\sum_{i\in S} x_i^2- 2kx_5^2+2x_1^2  \\
  &=k\left(1-2\sum_{ij\in E(H)}x_ix_j-4x_3^2\right)-2kx_5^2+2x_1^2,
  \end{align*}
  in which the last equality follows from the facts that $$|x_3|=|x_4|=|x_{n-3}|=|x_{n-2}|,\quad
    1=\|\x\|=2 \sum_{ij\in E(H)}x_ix_j+\sum_{i\in S} x_i^2+4x^2_3.$$
 Therefore,
  \begin{align} \label{eq:lambda2cubic}
  \sum_{ij\in E(H)}x_ix_j \leq \frac{k(1-4x_3^2-2x_5^2)+2x_1^2}{1+2k}.
  \end{align}
  From \eqref{eq:lambda1cubic} and \eqref{eq:lambda2cubic}, we obtain
\begin{equation}
\rho \geq -2-\frac{2k(1-4x_3^2-2x_5^2)+ 4x_1^2}{1+2k}-2x_5^2.
\end{equation}
By the eigen-equation, we see that
\begin{equation}\label{eq:x3x4}
x_3=\frac{1}{2}(\rho-1)x_1 ~~{\text{and}}~~ x_5=\frac{1}{2}(\rho^2-\rho-4)x_1,
\end{equation}
from which it follows that
\begin{equation}\label{eq:lambda3cubic}
\rho\geq -2-\frac{(\rho-1)^2(\rho^2-\rho-4)^2 x_1^2}{2(\rho^2-2\rho+5)}-\frac{4}{\rho^2-2\rho+5}.
\end{equation}
The facts that  $\|\x\|=1$, $n\geq 22$,
 $x_i^2=x_{n-i+1}^2$, $x_1=x_2$, and $x_7=x_8$ yield
 \begin{equation} \label{eq:x1x10}
 2x_1^2+2x_3^2+x_5^2+x_6^2+2x_7^2+x_{9}^2+x_{10}^2\leq\frac{1}{2}.
 \end{equation}
 Similarly to \eqref{eq:x3x4}, we can also write  $x_6, x_7, x_{9},$ and $x_{10}$ in terms of $x_1$.
 By plugging in all these into \eqref{eq:x1x10}, we obtain
\begin{equation}\label{eq:lambda4cubic}
x_1^2 \leq\frac8{f(\rho)},
\end{equation}
where
$$f(t)=t^{12}-4t^{11}-15t^{10}+64t^{9}+88t^{8}-364t^{7}-284t^{6}+840t^{5}+527t^{4}-584t^{3}-245 t^{2}-16t+280.$$
  Now from \eqref{eq:lambda3cubic} and \eqref{eq:lambda4cubic}, we get the following inequality in terms of $\rho$:
\begin{equation}\label{eq:g/f}
	\frac{g(\rho)}{f(\rho)\left(\rho^2-2\rho+5\right)}>0,
\end{equation}
  where
\begin{align*}
g(t)&=t^{15}-4t^{14}-14t^{13}+74t^{12}+17t^{11}-510 t^{10}+700t^{9}+1708t^{8}-4853t^{7}\\
&\quad-3716t^{6}+12026t^{5}+6770t^{4}-8061t^{3}-3474t^{2}-40t+3984.
\end{align*}
We observe that both $f(t)$ and $t^2-2t+5$ have no real zeros and so they are always positive. So \eqref{eq:g/f} implies that $g(\rho)>0$.
 Also $g(t)$ has a unique real zero which is greater than
$-2.406$. This means that if
$t\le-2.406$, then $g(t)<0$. It follows that  $\rho > -2.406> -1-\sqrt{2}$.
\end{proof}
Now, we are ready to prove the main result of this subsection.
 \begin{theorem}\label{thm:GapCubic}
 \begin{itemize}
 	\item[\rm(i)] The eigenvalues of $\de_n$ lying in the interval $\left[-3,-\sqrt{5}\right]$ converge to $(1-\sqrt{33})/2$ as $n$ tends to infinity. In particular,
 	$\underset{n\to \infty}{\lim}\rho(\de_n)=(1-\sqrt{33})/{2}$.
 		\item[\rm(ii)] $\de_n$ has no eigenvalue in the  interval $(1,\sqrt{5}]$.
 \end{itemize}
  \end{theorem}
 \begin{proof}
 Let  $\lambda \notin \lbrace -1, 0 \rbrace$ be an eigenvalue of  $\de_n$.
   So by the proof of Theorem~\ref{thm:simplecubic},  $\lambda$  has an eigenvector $\x$ that is constant on each cell of $\Pi$.
   The main components of $\x$ by Lemma~\ref{lem:Recursion}  are given by \eqref{eq:PQRS} and furthermore by Remark~\ref{remark2}   satisfy \eqref{eq:pm}.
   The proof involves showing that these two conditions simultaneously  for $\la\in(1,\sqrt{5}]$ do not hold and for $\la\in\left[-3,-\sqrt{5}\right]$ hold only if  $\la\to(1-\sqrt{33})/2$ as $n$ grows.

Since $\x$  satisfies \eqref{eq:pm}, $|a_{\frac{m}{2}}|=|a_{{\frac{m}{2}}+1}|$ if $m$ is even and
$|a_{\frac{m-1}{2}}|=|a_{\frac{m+3}{2}}|$ if $m$ is odd.
Suppose that $m=2l$ is even (if $m$ is odd, then  the argument is the same).
From \eqref{eq:PQRS}, we have
$$(P+1)R^{l}-(P-1)S^{l} =\pm( (P+1)R^{l+1}-(P-1)S^{l+1}).$$
Therefore, we have either
$$(P+1)R^{l}(1- R)=(P-1)S^{l}(1-S),\quad\text{or}\quad
(P+1)R^{l}(1+R)=(P-1)S^{l}(S+ 1).$$
We observe that $R$, $P-1$, and $1\pm R$ do not vanish as $\lambda \in I$.
Hence we obtain
\begin{equation}\label{eq:2eq}
\frac{P+1}{P-1}=\left(\frac{S}{R}\right)^{l} \frac{1-S}{1-R},\quad
  \text{or}\quad
  \frac{P+1}{P-1}=\left(\frac{S}{R}\right)^{l} \frac{1+S}{1+R}.
\end{equation}
Also since $\lambda \in I$, we see that
\begin{equation}\label{eq:RS}
0<\frac{S}{R}<1,\quad 0< \frac{1-S}{1-R}<1, ~\text{and}~ -1< \frac{1+S}{1+R}<0.
\end{equation}
Thus by \eqref{eq:2eq}, we have $\frac{P+1}{P-1}<1$.
On the other hand,
by direct computation using \eqref{eq:P} we observe that
for $\lambda \in (1,\sqrt{5})$, one has $\frac{P+1}{P-1}>1$. This implies that $\de_n$ has no eigenvalue in the interval $(1,\sqrt{5})$.
  Now let  $\lambda \in (-3,-\sqrt{5})$. From \eqref{eq:RS},  we have
 $\left(\frac{S}{R}\right)^{l}\frac{1\pm S}{1\pm R}\to 0$, as $l$ tends to infinity.
 So, by \eqref{eq:2eq},  we must have
 $P+1\to 0$, as $n$ tends to infinity. From \eqref{eq:P}, it runs out that
 $\left( \lambda-3 \right)  \left( {\lambda}^{2}-\lambda-8 \right)  \left( \lambda+1
 \right)^2
\to 0$. As $-3<\la<-\sqrt5$, this occurs only if $\lambda \to \frac{1-\sqrt{33}}{2}$ as $n$ tends to infinity.

To establish that $\lim_{n\to \infty}\rho(\de_n)=(1-\sqrt{33})/2$, it suffices to show that $\rho(\de_n)<-\sqrt5$. To see this, let $H$ be the induced subgraph on the first seven blocks of $\de_n$. It is seen that
 $\rho(H)<-\sqrt5$. Thus we are done by interlacing.

 It only remains to show that $\pm\sqrt5$  are not  eigenvalues of $\de_n$. For a contradiction,
 let $\lambda=\sqrt{5}$ be an eigenvalue with eigenvector $\x$. The components of $\x$ satisfy the   recurrence relation  \eqref{eq:Recursion} with  $k=-2$, $a_0=(\sqrt5-1)/2$ and $a_1=-(\sqrt5+3)/2$. It follows that
 $$a_j=\frac{(-1)^j}{2}\left(4j+\sqrt5-1\right).$$
 As before, we  may assume that $m=2l$ is even and so $|a_{l}|=|a_{l+1}|$.
If $a_{l}=a_{l+1}$, then $(-1)^{l}(4l+\sqrt5+1)=0$ and if $a_{l}=-a_{l+1}$, then $2(-1)^{l+1}=0$.
Both of which lead to contradictions, and so $\sqrt{5}$ is not an eigenvalue of $\de_n$. A similar argument works for $-\sqrt{5}$.
The proof is now complete.
 \end{proof}

We conjecture that the gap interval above is indeed maximal.
\begin{conjecture}\rm
	$(1,\sqrt{5}]$ is a maximal gap interval for the sequence $\de_n$.
\end{conjecture}

\begin{remark}\rm
	A referee asked  what the eigenfunction of the eigenvalues  converging to $(1-\sqrt{33})/2$ look like. Such eigenvalues are interesting as they lie in the gap of the infinite-graph spectrum.
Here we only consider the smallest eigenvalue, namely	$\rho=\rho(\de_n)$.  From  \eqref{eq:lambda3cubic}, we can obtain the following lower bound on the square of the first component of  a unit eigenvector for $\rho$:
$$x_1^2\ge\frac{-2(\rho^3+\rho+14)}{(\rho-1)^2(\rho^2-\rho-4)^2}.$$
As $\rho\to(1-\sqrt{33})/2$, we can conclude that $x_1^2>0.018$ for large enough $n$.
This shows that the components on the end blocks are of constant magnitude. In contrast, the components on the middle blocks (far away enough from the end blocks) seem to decay exponentially as $n$ grows; this can be justified using Lemma~\ref{lem:Recursion} but we do not pursue that here  (it can also be justified by more sophisticated but more general methods surveyed in \cite{aopBook}).
In Figure~\ref{fig:4262}, we plotted the (line charts of) unit eigenvectors for $\rho(\de_n)$ for few values of $n$, where the above properties of the components on end blocks and middle blocks can be noticed.
\end{remark}	
\begin{figure}
$$\begin{array}{ccccccc}
	\includegraphics[width=3.3cm]{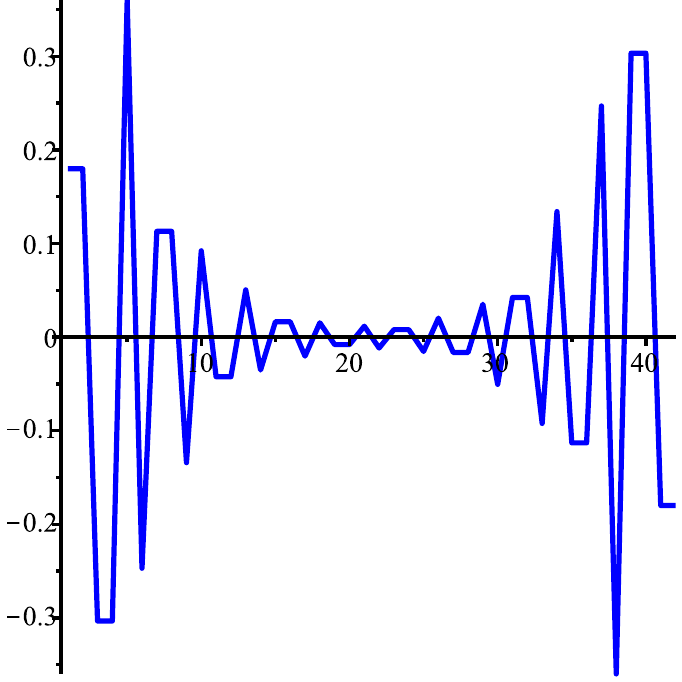}&&
	\includegraphics[width=3.3cm]{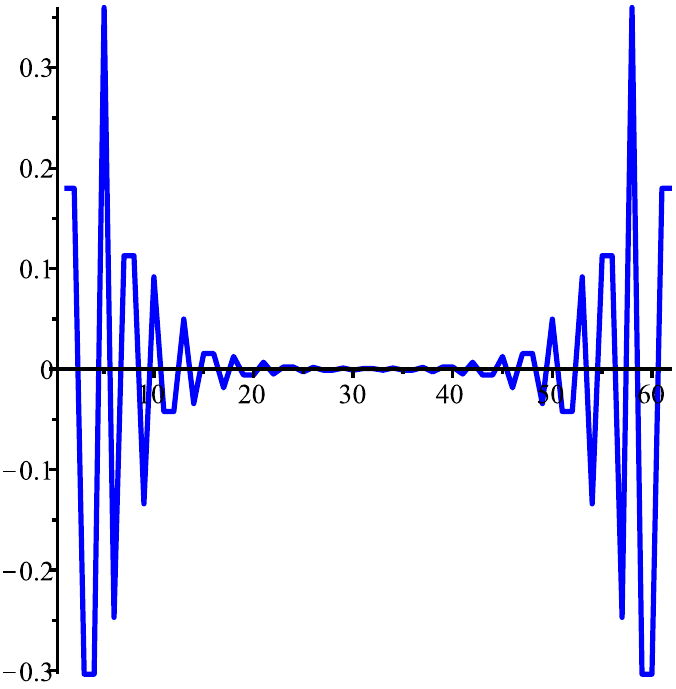}&&
	\includegraphics[width=3.3cm]{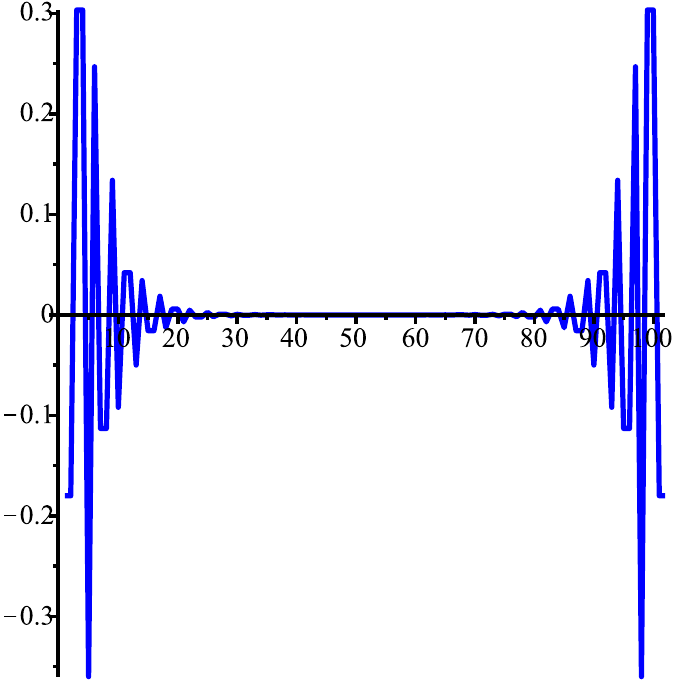}&&
	\includegraphics[width=3.3cm]{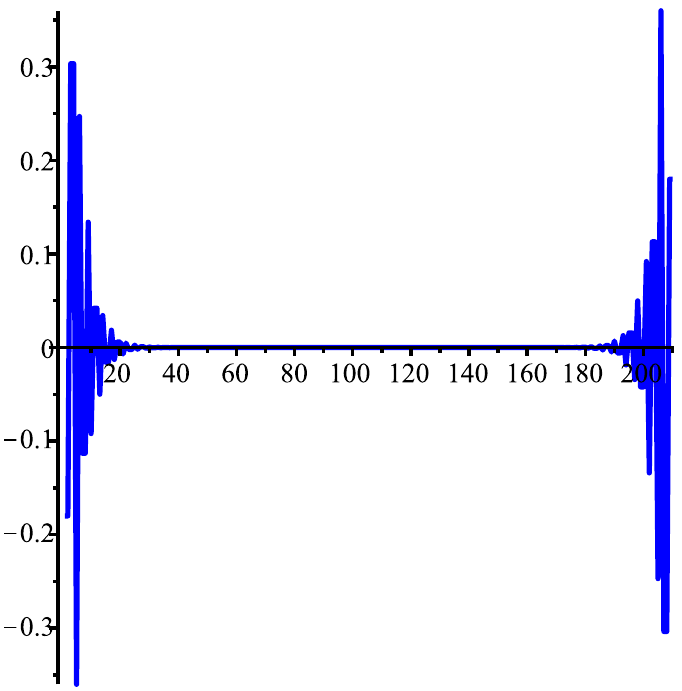}
\end{array}$$
\caption{The line charts of the components of unit eigenvectors for $\rho(\de_n)$ for $n=42,62,102,210$; the horizontal axis indicates the labels of vertices of $\de_n$  from left to right}
\label{fig:4262}
\end{figure}

\subsection{Gap intervals for the spectrum of $\ga_n$}
 \begin{remark} \label{remark3}\rm
 Let  $\lambda \notin\Lambda$ (see Theorem~\ref{thm:simple}) be an eigenvalue of  $\ga_n$ with eigenvector $\x=(x_1, \ldots, x_n)$.
By the symmetry of $\ga_n$, $\x'=(x_n, \ldots, x_1)$  is also an eigenvector for $\lambda$.
Since   $\lambda$ is simple (by Theorem~\ref{thm:simple}), it follows that  $\x=\pm\x'$, that is
\begin{equation}\label{eq:pm2}
 |x_i|=|x_{n-i+1}|, ~{\text{for}}~i=1, \ldots, n.
\end{equation}
 \end{remark}

We first establish a lower bound on the smallest eigenvalues of $\ga_n$.
\begin{theorem} \label{thm:Gapquartic1}
If $n\geq 11$, then   $\rho(\ga_n)\geq -1-\sqrt{3}$.
\end{theorem}
\begin{proof}
For convenience, let us set $G:=\ga_n$ and $\rho:=\rho(G)$.
If $n=11, 16,$ or $21$, then by direct computation we see that $\rho \geq -1-\sqrt{3}$.
So we suppose that $n\geq 26$.
Note that $\rho<-2.601<-1-\sqrt2$. This follows from interlacing and the fact that the smallest eigenvalue of the induced subgraph  on the first four blocks of $G$ is less than
 $-2.601$. So, by Theorem~\ref{thm:simple}, $\rho $ has a unit eigenvector $\x$ that is constant on each cell of $\Pi$ and by  Remark~\ref{remark3}, it satisfies \eqref{eq:pm2}.  The proof goes along the same line as the proof of Theorem~\ref{thm:Gapcubic1}
by first decomposing $A(G)$ as the sum of a positive definite matrix
and two other `simpler' matrices. This enable us to obtain an inequality in terms of a rational function in $\rho$ from which  the desired lower bound on $\rho$ will be derived. 

Consider a set of cliques ($K_2, K_3,$ and $K_4$) of $G$ that covers the edges of $G$,  as shown in Figure~\ref{fig:clique}.
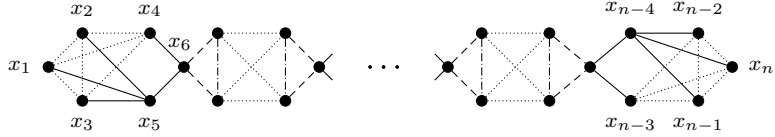
\begin{figure}[H]
\centering
\begin{tikzpicture}[scale=.9]
	\vertex[fill] (1) at (0,0) [label=left:\scriptsize{$x_1$}]{};
	\vertex[fill] (2) at (.5,.5) [label=above:\scriptsize{$x_2$}] {};
	\vertex[fill] (3) at (.5,-.5) [label=below:\scriptsize{$x_3$}] {};
	\vertex[fill] (4) at (1.5,.5) [label=above:\scriptsize{$x_4$}] {};
    \vertex[fill] (5) at (1.5,-.5) [label=below:\scriptsize{$x_5$}] {};
    \vertex[fill] (6) at (2,0) [label=above:\scriptsize{$x_6$\ \ }] {};
    \vertex[fill] (7) at (2.5,.5) [] {};
    \vertex[fill] (8) at (2.5,-.5) [] {};
    \vertex[fill] (9) at (3.5,.5) [] {};
    \vertex[fill] (10) at (3.5,-.5) [] {};
    \vertex[fill] (11) at (4,0) [] {};
    \vertex[fill] (22) at (5.9,0) [] {};
    \vertex[fill] (23) at (6.4,.5) [] {};
    \vertex[fill] (24) at (6.4,-.5) [] {};
    \vertex[fill] (25) at (7.4,.5) [] {};
    \vertex[fill] (26) at (7.4,-.5) [] {};
    \vertex[fill] (27) at (8,0) [] {};
    \vertex[fill] (28) at (8.6,.5) [label=above:\scriptsize{$x_{n-4}$}] {};
    \vertex[fill] (29) at (8.6,-.5) [label=below:\scriptsize{$x_{n-3}$}] {};
    \vertex[fill] (30) at (9.6,.5) [label=above:\scriptsize{$x_{n-2}$}] {};
    \vertex[fill] (31) at (9.6,-.5) [label=below:\scriptsize{$x_{n-1}$}] {};
    \vertex[fill] (32) at (10.1,0) [label=right:\scriptsize{$x_{n}$}] {};
    \tikzstyle{vertex}=[circle, draw, inner sep=0pt, minimum size=0pt]
    \vertex[fill] (12) at (4.2,.2)[] {};
    \vertex[fill] (13) at (4.2,-.2)[] {};
     \vertex[fill] (20) at (5.7,.2)[] {};
     \vertex[fill] (21) at (5.7,-.2)[] {};
    \tikzstyle{vertex}=[circle, draw, inner sep=0pt, minimum size=1pt]
    \vertex[fill] (15) at (4.75,0)[]{} ;
    \vertex[fill] (17) at (4.95,0)[]{} ;
    \vertex[fill] (19) at (5.15,0)[]{} ;
	\path[densely dotted]
		(1) edge (2)
		(1) edge (3)
	    (1) edge (4)
	    (2) edge (3)
		(2) edge (4)
	    (3) edge (4);
	    \path
		(1) edge (5)
	    (2) edge (5)	
	    (3) edge (5)
	    (4) edge (6)
	    (5) edge (6);
	 \path[densely dashed]
	    (6) edge (7)
	    (6) edge (8);
	  \path[densely dashdotted]
	    (7) edge (8);
	\path[densely dotted]
	    (7) edge (9)
	    (7) edge (10)
	    (8) edge (9)
	    (8) edge (10);
	  \path[densely dashdotted]
	    (9) edge (10);
	\path[densely dashed]
	    (9) edge (11)
	   (10) edge (11);
	   \path
	   (11) edge (12)
	   (11) edge (13)
	   (20) edge (22)
	   (21) edge (22);
	  \path[densely dashed]
	   (22) edge (23)
	   (22) edge (24);
	 \path[densely dashdotted]
	   (23) edge (24);
	\path[densely dotted]
	   (23) edge (25)
	   (23) edge (26)
	   (24) edge (25)
	   (24) edge (26);
	   \path[densely dashdotted]
	   (25) edge (26);
	    \path[densely dashed]
	    (25) edge (27)
	    (26) edge (27);
	    \path
	   (27) edge (28)
	   (27) edge (29)
	   (28) edge (30)
	    (28) edge (31)
	    (28) edge (32);
	  \path[densely dotted]
	   (29) edge (30)
	   (29) edge (31)
	   (29) edge (32)
	    (30) edge (31)
	    (30) edge (32)
	   (31) edge (32) ;
\end{tikzpicture}
\caption{The edge clique cover of $\ga_n$: each normal line indicates a $K_2$, each dashed triangle is a $K_3$,  each dotted tetrahedral is a $K_4$; dash-dotted edges are covered by a $K_3$ and a $K_4$.} \label{fig:clique}
\end{figure}
Let $H$ be the spanning subgraph of $G$ induced by the edges that are covered by two cliques, see Figure~\ref{fig:H}.
 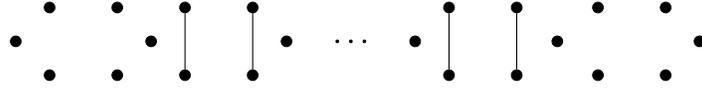
\begin{figure}[h!]
\centering
\begin{tikzpicture}[scale=.9]
	\vertex[fill] (1) at (0,0) []{};
	\vertex[fill] (2) at (.5,.5) [] {};
	\vertex[fill] (3) at (.5,-.5) [] {};
	\vertex[fill] (4) at (1.5,.5) [] {};
    \vertex[fill] (5) at (1.5,-.5) [] {};
    \vertex[fill] (6) at (2,0) [] {};
    \vertex[fill] (7) at (2.5,.5) [] {};
    \vertex[fill] (8) at (2.5,-.5) [] {};
    \vertex[fill] (9) at (3.5,.5) [] {};
    \vertex[fill] (10) at (3.5,-.5) [] {};
    \vertex[fill] (11) at (4,0) [] {};
    \vertex[fill] (22) at (5.9,0) [] {};
    \vertex[fill] (23) at (6.4,.5) [] {};
    \vertex[fill] (24) at (6.4,-.5) [] {};
    \vertex[fill] (25) at (7.4,.5) [] {};
    \vertex[fill] (26) at (7.4,-.5) [] {};
    \vertex[fill] (27) at (8,0) [] {};
    \vertex[fill] (28) at (8.6,.5) [] {};
    \vertex[fill] (29) at (8.6,-.5) [] {};
    \vertex[fill] (30) at (9.6,.5) [] {};
    \vertex[fill] (31) at (9.6,-.5) [] {};
    \vertex[fill] (32) at (10.1,0) [] {};
    \tikzstyle{vertex}=[circle, draw, inner sep=0pt, minimum size=1pt]
    \vertex[fill] (15) at (4.75,0)[]{} ;
    \vertex[fill] (17) at (4.95,0)[]{} ;
    \vertex[fill] (19) at (5.15,0)[]{} ;
	\path
	    (7) edge (8)
	    (9) edge (10)
	    (23) edge (24)
	    (25) edge (26);
\end{tikzpicture}
\caption{The subgraph $H$ induced by the edges covered by two cliques} \label{fig:H}
\end{figure}

 Let $M$ be the vertex-clique incidence matrix of $G$. Also let $A$ and $B$ be  the adjacency matrices of $G$ and $H$, respectively, and $C$ be the matrix whose entries are all zero except for $C_{5, 5}=C_{n-4, n-4}=2$ and $C_{6, 6}=C_{n-5, n-5}=1$.
  Note that the $(i,j)$-entry of $MM^\top$ is equal to the inner product of the two rows of $M$ corresponding to the vertices $i$ and $j$. Each vertex of $G$ belongs to two cliques of the clique covering except for the vertices $5$ and $n-4$ that belong to four cliques and the vertices $6$ and $n-5$ that belong to three cliques.
   This implies that the diagonal entries of $MM^\top$ and $2I+C$ are the same. Now, if $ij \notin  E(G)$, then
   $i$ and $j$  together do not appear in any clique. If $ij \in E(G)\backslash E(H)$, then $i$ and $j$ appear in exactly one clique together and if $ij \in E(H)$, they appear in two cliques together. From this argument, it follows that the off-diagonal entries of $MM^\top$ coincide with those of $A+B$.  So
$$MM^\top=2I+A+B+C.$$
Therefore,
$$\rho=\x A\x^\top=(\x M)(\x M)^\top-2-\x B\x^\top-\x C\x^\top.$$
We have $\x B\x^\top=2\sum_{ij\in E(H)}x_ix_j$ and since $|x_5|=|x_{n-4}|$ and $|x_6|=|x_{n-5}|$,
$\x C\x^\top=4x_5^2+2x_6^2$. It follows that
\begin{equation}\label{eq:lambda}
\rho \geq -2-2\sum_{ij\in E(H)}x_ix_j-4x_5^2-2x_6^2.
\end{equation}
Note that when $ij\in E(H)$, then $x_ix_j=x_i^2$.
Suppose that $x_r, x_s, x_t,$ and $x_u$ are the components of $\x$ on the cells of a middle block of $G$, as depicted in Figure~\ref{fig:1middlequartic}.
\begin{figure}[H]
\centering
\begin{tikzpicture}[scale=0.9]
	\vertex[fill] (r) at (0,0) [label=above:\footnotesize{$x_r$}] {};
	\vertex[fill] (r1) at (.5,.5) [label=above:\footnotesize{$x_s$}] {};
	\vertex[fill] (r2) at (.5,-.5) [] {};
	\vertex[fill] (r3) at (1.5,.5) [label=above:\footnotesize{$x_t$}] {};
	\vertex[fill] (r4) at (1.5,-.5) [] {};
    \vertex[fill] (r5) at (2,0) [label=above:\footnotesize{$x_u$}] {};
	\path
		(r) edge (r1)
		(r) edge (r2)
		(r1) edge (r2)
		(r1) edge (r3)
        (r1) edge (r4)
		(r2) edge (r3)
	    (r2) edge (r4)
	    (r3) edge (r4)
	    (r5) edge (r4)
	    (r3) edge (r5);
\end{tikzpicture}
\caption{A middle block of  $G$ and the components of $\x$} \label{fig:1middlequartic}
\end{figure}
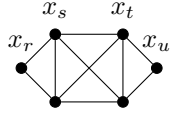

By using the eigen-equation on the vertices $s$ and $t$, we have
$$x_s=\frac{(\rho-1)x_r+2x_u}{\rho^2-2\rho-3},\quad \text{and} \quad x_t=\frac{(\rho-1)x_u+2x_r}{\rho^2-2\rho-3}.$$
It turns out that $x_s^2+x_t^2\leq k(x_r^2+x_u^2)$ in which $k=\frac{2(\rho-1)^2+8}{(\rho^2-2\rho-3)^2}$. Let $S$ be the set of  cut vertices of $G$. Then
  \begin{align*}
  \sum_{ij\in E(H)}x_ix_j &\leq 2k\sum_{i\in S} x_i^2- 2kx_6^2  \\
  &=2k\left(1-2\sum_{ij\in E(H)}x_ix_j-6x_1^2-4x_4^2\right)-2kx_6^2,
  \end{align*}
where the last equality follows from the fact that
  $$1=\|\x\|=2 \sum_{ij\in E(H)}x_ix_j+\sum_{i\in S} x_i^2+6x^2_1+4x^2_4.$$
 Therefore,
  \begin{align} \label{eq:lambdaa}
  \sum_{ij\in E(H)}x_ix_j \leq \frac{2k\left(1-6x_1^2-4x_4^2-x_6^2\right)}{1+4k}.
  \end{align}
Combining \eqref{eq:lambda} and \eqref{eq:lambdaa}, we get
  $$\rho\geq -2-\frac{4k(1-6x_1^2-4x_4^2-x_6^2)}{1+4k}-4x_4^2-2x_6^2.$$
Using the eigen-equation, we can write $x_4,x_6$ in terms of $x_1$ as
  \begin{equation*}
  x_4=\frac{1}{2}(\rho -2)x_1, \quad\text{and}\quad x_6=\frac{1}{2}(\rho^2-2\rho-6)x_1.
  \end{equation*}
We substitute these and the value of $k$ in terms of $\rho$ in the last inequality to obtain
  \begin{equation}\label{eq:lambdaaa}
  \rho \geq-2+\frac{f(\rho)x_1^2-16\rho^2+32
\rho-80}{2\rho^4-8\rho^3+12\rho^2-8\rho+98},
\end{equation}
where
$$f(t)=-t^8+8t^7-12t^6-36t^5+27t^4+236t^3+254t^2-
1056t-636.$$
The smallest zero of $f$ is larger than $-2.034$.
As $\rho<-1-\sqrt2$,  $f(\rho)$, i.e. the coefficient of $x_1^2$ in \eqref{eq:lambdaaa}, is negative.
On the other hand, using the facts that $n\geq 26$, $\|\x\|=1$,  $x_i^2=x_{n-i+1}^2$  for $i=1, \ldots, n$,
 and  $\x$ being constant on each cell of $\Pi$, we can obtain that
 \begin{equation}\label{eq:x1x12}
 3x_1^2+2x_4^2+x_6^2+2x_7^2+2x_9^2+x_{11}^2+2x_{12}^2\leq\frac{1}{2}.
 \end{equation}
We can also write  $x_7, x_9, x_{11},$ and $x_{12}$ in terms of $x_1$ by using the eigen-equation.
 Plugging in all these into \eqref{eq:x1x12} yields
\begin{equation}\label{eq:lambdaaaa}
x_1^2\le\frac{64}{g(\rho)},
\end{equation}
where
$$g(t)=t^{12}-8t^{11}-4t^{10}+148t^9-103t^8-1020t^7+714t^6
+3024t^5-960t^4-2464t^3+1024t^2+3712.$$
Now from \eqref{eq:lambdaaa} and \eqref{eq:lambdaaaa}, we come up with
\begin{equation}\label{eq:h/g}
\frac{h(\rho)}{g(\rho)\left(\rho
	^4-4\rho^3+6\rho^2-4\rho+49 \right)}>0,\end{equation}
where
\begin{align*}
h(t)&=	t^{17}-\!10t^{16}+\!10t^{15}+\!188t^{14}-\!494t^{13}-\!1236t^{12}+\!4124t^{11}+\!5136t^{10}\!-6907t^9 -\!34850t^8\\
	&\quad-\!66910t^7+\!162036t^6+\!
	356704t^5-\!185984t^4-\!329408t^3+\!192576t
	^2+\!126592t+\!532608.
\end{align*}
We observe that both $g(t)$ and $t^4-4t^3+6t^2-4t+49$ have no real zeros and so they are positive for any real $t$. So \eqref{eq:h/g} implies that $h(\rho)>0$.
Note that $h(t)$ has a unique real zero which is greater than
$-2.71$. This means that if
$t\le-2.71$, then $h(t)<0$. It follows that  $\rho > -2.71> -1-\sqrt3$.
\end{proof}

Now, we are ready to prove the main result of this subsection.
 \begin{theorem}
\begin{itemize}
\item[\rm(i)]  The eigenvalues $\ga_n$ lying in the interval $[-4,-(1+\sqrt{17})/2]$ converge to $1-\sqrt{13}$ as $n$ tends to infinity. In particular,
 $\underset{n\to \infty}{\lim}\rho(\ga_n)=1-\sqrt{13}$.
\item[\rm(ii)] $\ga_n$ has no eigenvalue in the interval $[(-1+\sqrt{17})/2,3]$.
\end{itemize}
 \end{theorem}
 \begin{proof}
 Let  $\lambda \notin\Lambda$ (given in Theorem~\ref{thm:simple}) be an eigenvalue of  $\ga_n$.
 So, by the proof of Theorem~\ref{thm:simple}, $\lambda$  has an eigenvector $\x$ that is constant on each cell and by Remark~\ref{remark3}, it satisfies \eqref{eq:pm2}.
 We can take $x_1=1$ (note that $x_1\neq 0$ since otherwise by the eigen-equation $\x=\bf0$).
 The proof goes along the same line as the proof of Theorem~\ref{thm:GapCubic}. We first obtain a similar relation as in \eqref{eq:PQRS} for the main components of $\x$. Then we put together such a relation with \eqref{eq:pm2} and show these impose the eigenvalues to satisfy the statement.	 
 
 Consider two consecutive middle blocks of $\ga_n$, as depicted in Figure~\ref{fig:3middlequartic},  in which the  labels of the vertices indicate the components of $\x$.
  \begin{figure}[h!]
	\centering
	\begin{tikzpicture}[scale=.9]
	\vertex[fill] (1) at (0,0) [label=above:\scriptsize{$a_{i-2}\ \ $}] {};
	\vertex[fill] (2) at (.6,.5) [label=above:\scriptsize{$b_{i-2}$}] {};
	\vertex[fill] (3) at (.6,-.5) [] {};
	\vertex[fill] (4) at (1.7,.5) [label=above:\scriptsize{$c_{i-2}$}] {};
	\vertex[fill] (5) at (1.7,-.5) [] {};
	\vertex[fill] (6) at (2.3,0) [label=above:\scriptsize{$a_{i-1}$}] {};
	\vertex[fill] (7) at (2.9,.5) [label=above:\scriptsize{$b_{i-1}$}] {};
	\vertex[fill] (8) at (2.9,-.5) [] {};
	\vertex[fill] (9) at (4,.5) [label=above:\scriptsize{$c_{i-1}$}] {};
	\vertex[fill] (10) at (4,-.5) [] {};
	\vertex[fill] (11) at (4.6,0) [label=above:\scriptsize{$a_{i}$}] {};
	\path
(3) edge (2)
(4) edge (5)
(5) edge (2)
(3) edge (4)
	(1) edge (2)
	(1) edge (3)
	(3) edge (5)
	(2) edge (4)
	(4) edge (6)
	(5) edge (6)
	(6) edge (7)
	(6) edge (8)
(7) edge (8)
(7) edge (10)
(8) edge (9)
(9) edge (10)
	(7) edge (9)
	(8) edge (10)
	(10) edge (11)
	(9) edge (11);
	\end{tikzpicture}
	\caption{Three middle blocks of  $\ga_n$ and the components of $\x$}\label{fig:3middlequartic}
\end{figure}

 Using the eigen-equation, we obtain
$$\begin{array}{rrr}
(\lambda-1) b_{i-2}- a_{i-2}-2c_{i-2}=0,& (\lambda-1) c_{i-2}-2b_{i-2}-a_{i-1}=0,& \lambda a_{i-1}-2c_{i-2}-2b_{i-1}=0, \\
(\lambda-1) b_{i-1}-a_{i-1}-2c_{i-1}=0,& (\lambda-1)  c_{i-1}-2b_{i-1}-a_i=0.&
\end{array}$$
 From these equations we can  write $a_{i-1}$ in terms of $a_{i-2}$ and  $a_{i}$ as:
\begin{equation}\label{eq:ai-1quartic}
a_{i-1}=\frac{4(a_{i-2}+a_{i})}{\lambda^3-2\lambda^2-7\lambda+4}.
\end{equation}
Let $a_0, \ldots, a_m$ be the components of $\x$ on the cut vertices of $\ga$, as shown in Figure~\ref{fig:aquartic}.
\begin{figure}[H]
\centering
\begin{tikzpicture}[scale=.9]
	\vertex[fill] (1) at (0,0) [label=left:\footnotesize{$1$}]{};
	\vertex[fill] (2) at (.5,.5) [label=above:\footnotesize{$1$}] {};
	\vertex[fill] (3) at (.5,-.5) [] {};
	\vertex[fill] (4) at (1.5,.5) [] {};
    \vertex[fill] (5) at (1.5,-.5) [] {};
    \vertex[fill] (6) at (2,0) [label=above:\footnotesize{$a_0$}] {};
    \vertex[fill] (7) at (2.5,.5) [] {};
    \vertex[fill] (8) at (2.5,-.5) [] {};
    \vertex[fill] (9) at (3.5,.5) [] {};
    \vertex[fill] (10) at (3.5,-.5) [] {};
    \vertex[fill] (11) at (4,0) [label=above:\footnotesize{$a_1$}] {};
    \vertex[fill] (22) at (5.9,0) [] {};
    \vertex[fill] (23) at (6.4,.5) [] {};
    \vertex[fill] (24) at (6.4,-.5) [] {};
    \vertex[fill] (25) at (7.4,.5) [] {};
    \vertex[fill] (26) at (7.4,-.5) [] {};
    \vertex[fill] (27) at (8,0) [label=above:\footnotesize{$a_m$}] {};
    \vertex[fill] (28) at (8.6,.5) [] {};
    \vertex[fill] (29) at (8.6,-.5) [] {};
    \vertex[fill] (30) at (9.6,.5) [] {};
    \vertex[fill] (31) at (9.6,-.5) [] {};
    \vertex[fill] (32) at (10.1,0) [] {};
    \tikzstyle{vertex}=[circle, draw, inner sep=0pt, minimum size=0pt]
    \vertex[fill] (12) at (4.2,.2)[] {};
    \vertex[fill] (13) at (4.2,-.2)[] {};
     \vertex[fill] (20) at (5.7,.2)[] {};
    \vertex[fill] (21) at (5.7,-.2)[] {};
    \tikzstyle{vertex}=[circle, draw, inner sep=0pt, minimum size=1pt]
    \vertex[fill] (15) at (4.75,0)[]{} ;
    \vertex[fill] (17) at (4.95,0)[]{} ;
    \vertex[fill] (19) at (5.15,0)[]{} ;
	\path
		(1) edge (2)
		(1) edge (3)
	    (1) edge (4)
	    (2) edge (3)
		(2) edge (4)
	    (3) edge (4)
		(1) edge (5)
	    (2) edge (5)	
	    (3) edge (5)
	    (4) edge (6)
	    (5) edge (6)
	    (6) edge (7)
	    (6) edge (8)
	    (7) edge (8)
	    (7) edge (9)
	    (7) edge (10)
	    (8) edge (9)
	    (8) edge (10)
	    (9) edge (10)
	    (9) edge (11)
	   (10) edge (11)
	   (11) edge (12)
	   (11) edge (13)
	   (20) edge (22)
	    (21) edge (22)
	   (22) edge (23)
	   (22) edge (24)
	   (23) edge (24)
	   (23) edge (25)
	   (23) edge (26)
	   (24) edge (25)
	   (24) edge (26)
	   (25) edge (26)
	    (25) edge (27)
	    (26) edge (27)
	   (27) edge (28)
	   (27) edge (29)
	   (28) edge (30)
	    (28) edge (31)
	    (28) edge (32)
	   (29) edge (30)
	   (29) edge (31)
	   (29) edge (32)
	    (30) edge (31)
	    (30) edge (32)
	   (31) edge (32) ;
\end{tikzpicture}
\caption{The quartic graph $\ga_n$ with $m$ middle blocks}
\label{fig:aquartic}
\end{figure}
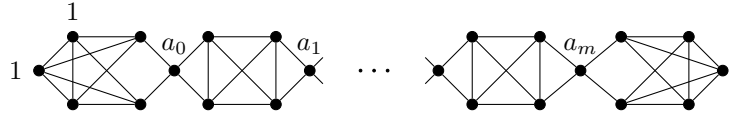 
From \eqref{eq:ai-1quartic}, we obtain the following recurrence relation on the components of $\x$:
\begin{equation}\label{eq:Recursionquartic}
 \begin{array}{l}
 	a_i=k a_{i-1}-a_{i-2},\quad i=2,\ldots,m,\\  a_0=\frac{1}{2}(\lambda^2-2\lambda-6),\quad a_1=\frac{1}{8}(\lambda^5-4\lambda^4-9\lambda^3+32\lambda^2+24\lambda-24),
 	\end{array}
\end{equation}
 where  $k=(\lambda^3-2\lambda^2-7\lambda+4)/4$.
 Note that $a_0$ and $a_1$ are obtained using the eigen-equation on the first five cells of $\ga_n$.
To solve the recurrence relation \eqref{eq:Recursionquartic}, we find the zeros of its characteristic equation
 $x^2-kx+1=0$, that is
  $$R=(k-\sqrt{k^2-4})/2,\quad S=(k+\sqrt{k^2-4})/2.$$
 Now we assume that  $\lambda\in I:=\left(-4,-(1+\sqrt{17})/2\right)\cup \left((-1+\sqrt{17})/2,3\right)$.
So  $k^2-4 >0$ and so $R,S$ are reals.
  It turns out that the solution of \eqref{eq:Recursionquartic} is
\begin{equation}\label{eq:PQRSquartic}
  a_j=\frac{1}{4}\left( (P+Q)R^j-(P-Q)S^j\right),
\end{equation}
in which
\begin{equation}\label{eq:P,Q}
	P=\frac{-\lambda^5+4\lambda^4+9\lambda^3-34\lambda^2-14\lambda+24}{4\sqrt {k^2-4}},
\quad
Q=\lambda^2-2\lambda-6.\end{equation}
Since $\x$  satisfies \eqref{eq:pm2}, $|a_{\frac{m}{2}-1}|=|a_{{\frac{m}{2}}+1}|$ if $m$ is even and
$|a_{\frac{m-1}{2}}|=|a_{\frac{m+1}{2}}|$ if $m$ is odd.
We suppose that $m=2l$ is even (if $m$ is odd, then  the argument is the same).
So from \eqref{eq:PQRSquartic}, we have
$$(P+Q)R^{l-1}-(P-Q)S^{l-1} =\pm\left( (P+Q)R^{l+1}-(P-Q)S^{l+1}\right).$$
Therefore, we have either
$$(P+Q)R^{l-1}(1- R^2)=(P-Q)S^{l-1}(1-S^2),\quad
\text{or}\quad
(P+Q)R^{l-1}(1+R^2)=(P-Q)S^{l-1}(1+S^2).$$
Note that $R$, $S$, $1\pm R^2$, and $1\pm S^2$ do not vanish since $\lambda \in I$.

We further assume that $\lambda \in \left((-1+\sqrt{17})/2,3\right)$. So $P+Q\neq 0$ and we have either
\begin{equation}\label{eq:2eqquartic}
\frac{P-Q}{P+Q}=\left(\frac{R}{S}\right)^{l-1} \frac{1-R^2}{1-S^2},\quad
  \text{or}\quad
 \frac{P-Q}{P+Q}=\left(\frac{R}{S}\right)^{l-1} \frac{1+R^2}{1+S^2}.
\end{equation}
Also
\begin{equation}\label{eq:R/S>1}
\frac{R}{S}>1,\quad   \frac{1-R^2}{1-S^2}<-1, ~\text{and}\quad  \frac{1+R^2}{1+S^2}>1.
\end{equation}
So by \eqref{eq:2eqquartic}, we have that  $\left|\frac{P-Q}{P+Q}\right|>1$.
However, from \eqref{eq:P,Q} it is seen that $-1<\frac{P-Q}{P+Q}<1$, a contradiction.
 Therefore, $\ga_n$ has no eigenvalue in the interval $\left((-1+\sqrt{17})/2,3\right)$.

 Next, assume that $\lambda \in \left(-4, (-1-\sqrt{17})/2\right)$. So $P-Q\neq 0$ and we have
\begin{equation}\label{eq:2eqqquartic}
\frac{P+Q}{P-Q}=\left(\frac{S}{R}\right)^{l-1} \frac{1-S^2}{1-R^2},\quad
  \text{or}\quad
 \frac{P+Q}{P-Q}=\left(\frac{S}{R}\right)^{l-1} \frac{1+S^2}{1+R^2}.
\end{equation}
Also
\begin{equation*}
0<\frac{S}{R}<1,~ -1<\frac{1-S^2}{1-R^2}<0, ~{\text{and}}~  0<\frac{1+S^2}{1+R^2}<1.
\end{equation*}
Thus
 $\left(\frac{S}{R}\right)^{l-1}\frac{1 \pm S^2}{1\pm R^2}\to 0$, as $l$ tends to infinity.
 From \eqref{eq:2eqqquartic}, it then follows that
 $P+Q\to0$, where
 $$P+Q=\la^2-2\la-6-\frac {\la^5-4\la^4-9\la^3+34\la^2+14\la-24}{4\sqrt {k^2-4 }}.$$
Therefore,
 $( \lambda-2)( \lambda-3 ) ( \lambda-4)
 ( {\lambda}^{2}-2\,\lambda-12 ) ( \lambda+1) ^{2}
\to 0$.
As $-4<\la<-(1+\sqrt{17})/2$, we must have
 $\lambda \to {1-\sqrt{13}}$ as $n$ tends to infinity.

 To establish that $\lim_{n\to \infty}\rho(\ga_n)=1-\sqrt{13}$, it suffices to show that $\rho(\ga_n)<-(1+\sqrt{17})/2$. This is already done in the proof of Theorem~\ref{thm:Gapquartic1}, where we showed that $\rho(\ga_n)<-2.601<-(1+\sqrt{17})/2$.

It only remains to show that $(-1\pm\sqrt{17})/2$ and $3$ are not eigenvalues of $\ga_n$. For a contradiction,
 let $(-1+\sqrt{17})/2$ be an eigenvalue of $\ga_n$ with eigenvector $\x$.
  The components of $\x$ satisfy the   recurrence relation  \eqref{eq:Recursionquartic} with
     $k=-2$, $a_0=(-3\sqrt{17}-1)/4$, and $a_1=(3\sqrt{17}+9)/4$. It turns out that
 $$a_j=\frac{(-1)^{j+1}}{4}\left(3\sqrt{17}+1+8j\right).$$
 As before, we  may assume that $m=2l$ is even and $|a_{l-1}|=|a_{l+1}|$.
If $a_{l-1}=a_{l+1}$, then $4(-1)^{l+1}=0$ and if $a_{l-1}=-a_{l+1}$, then 
 $ (-1)^{l}\left(3\sqrt{17}+1+8l\right)=0$. 
Both of which lead to contradictions, and so $(-1+\sqrt{17})/2$ is not an eigenvalue of $\ga_n$. Similar arguments work for $(-1-\sqrt{17})/2$ and $3$.  So  $\ga_n$ has no eigenvalue
in the interval $[(-1+\sqrt{17})/2,3]$.
 The proof is now complete.
 \end{proof}
 We conjecture that the gap interval above is indeed maximal.

\begin{conjecture}\rm
$[\frac{-1+\sqrt{17}}{2},3]$ is a maximal gap interval for the sequence $\ga_n$.
\end{conjecture}

\section*{Acknowledgements}
The authors thank Peter Sarnak for insightful comments.
They are also indebted to an anonymous referee whose detailed and helpful comments greatly improved the exposition of the paper.
The second author carried out this work during a Humboldt Research Fellowship at the University of Hamburg. He thanks the Alexander von Humboldt-Stiftung for financial support.

\end{document}